\documentclass[reqno]{amsart}
\usepackage{graphicx}
\usepackage{amscd}
\usepackage{amsmath}
\usepackage{amsfonts}
\usepackage{amssymb}
\usepackage{cite}
\usepackage{xcolor}
\def\R {\mathbb{R}}
\def\C {\mathcal{C}}
\newtheorem{proposition}{Proposition}[section]
\newtheorem{theorem}[proposition]{Theorem}
\newtheorem{corollary}{Corollary}[section]
\newtheorem{lemma}{Lemma}[section]
\theoremstyle{definition}
\newtheorem{definition}{Definition}[section]
\newtheorem{remark}{Remark}[section]
\numberwithin{equation}{section}
\newtheorem{example}{Example}[section]

\begin{document}

\title[Longtime behavior of multi-term subdiffusion]
{Longtime behavior of  semilinear multi-term\\
fractional in time diffusion}

\author[ N. Vasylyeva]
{Nataliya Vasylyeva}

\address{Institute of Applied Mathematics and Mechanics of NAS of Ukraine
\newline\indent
G.Batyuka st.\ 19, 84100 Sloviansk, Ukraine; and
\newline\indent
Dipartimento di Matematica, Politecnico di Milano
\newline\indent
Piazza Leonardo da Vinci 32, 20133 Milano, Italy}
\email[N.Vasylyeva]{nataliy\underline{\ }v@yahoo.com}

\subjclass[2000]{Primary 35R11,35B45; Secondary 35B65, 26A33,35Q92}
\keywords{semilinear multi-term subdiffusion,  Caputo derivative,
longtime behavior, absorbing sets}

\begin{abstract}
In the paper, the initial-boundary value problems to a semilinear
integro-differential equation with multi-term fractional Caputo
derivatives are analyzed. A particular case of this equation models
oxygen diffusion through capillaries. Under proper assumptions on
the coefficients and a nonlinearity, the longtime behavior (as
$t\to+\infty$) of a solution is discussed. In particular, the
existence of  absorbing sets in suitable functional spaces is
established.
\end{abstract}

\maketitle

\section{Introduction}
\label{s1}

\noindent The key feature of fractional differential equations is
their effectiveness in the description of memory or delay phenomena,
which are structurally present in real-life models such as
underwound sewage survey, oil pollution survey, tumor growth, oxygen
subdiffusion \cite{AG,FKRWS,HH,MGSKA}. We are also refer to the
monograph \cite[Chapter 1]{CDV}, where the authors discuss several
other situations in which anomalous diffusion and fractional
differential equations naturally appear. A multi-term fractional in
time diffusion equation (\textbf{MTFDE}) is a special diffusion
equation where a partial derivative in time is replaced by the
fractional differential operator (\textbf{FDO}), $\mathbf{D}_{t}$,
written as a linear combination of fractional derivatives. For
example,
\begin{equation}\label{1.6}
\mathbf{D}_{t}=q\mathbf{D}_{t}^{\nu}+\sum_{i=1}^{N}q_i\mathbf{D}_{t}^{\nu_i},
\end{equation}
where $q_i$, $q$ are positive, and $0\leq
\nu_1<\nu_2<...<\nu_{M}<\nu<1$, and the symbol
$\mathbf{D}_{t}^{\theta}$ stands for the Caputo fractional
derivative with respect to time of order $\theta\in(0,1)$,
\[
\mathbf{D}_{t}^{\theta}u(x,t)=
\begin{cases}
\frac{1}{\Gamma(1-\theta)}\frac{\partial}{\partial t}\int\limits_{0}^{t}\frac{u(x,s)-u(x,0)}{(t-s)^{\theta}}ds\quad\text{if}\quad \theta\in(0,1),\\
\frac{\partial u}{\partial t}(x,t)\qquad\qquad\qquad\qquad\,\text{
if}\quad \theta=1,
\end{cases}
\]
where $\Gamma$ is the Euler Gamma-function.

 We recall that the
peculiarity of anomalous diffusion described with \textbf{MFTDEs} is
that the mean-squared displacement behaving as $\langle\Delta
x^{2}\rangle\sim t^{\min\{\nu_{i}\}}$ as $t\to+\infty$ (see e.g.
\cite{MMPG}) has sublinear growth in time. It is worth noting that
multi-term fractional diffusion equations
 improve the modeling accuracy of the single-term
model for describing anomalous phenomena. For example, in
\cite{SBMB}, a two-term fractional order diffusion model is proposed
for the total concentration in solute transport, in order to
distinguish explicitly the mobile and immobile status of the solute
exploiting fractional dynamics. The kinetic equation with two
fractional derivatives arises  also quite naturally in describing
subdiffusive motion in velocity fields \cite{MKS}; see also
\cite{KGM} for discussions on the model related to   wave-type
processes.

In this art, motivated by the discussion above, we focus on the
analytical study of initial-boundary value problems to the
semilinear integro-differential  equation with memory terms.
 Let $\Omega\subset\R^{n},$ $n\geq 1,$ be a bounded domain with a smooth
boundary $\partial\Omega,$ and for any  $T>0,$ we set
$$
\Omega_{T}=\Omega\times (0,T)\qquad \text{and}\qquad
\partial\Omega_{T}=\partial\Omega\times [0,T].
$$
We consider the initial-value problem to the nonautonomous
multi-term time-fractional semilinear diffusion equation in the
unknown function $u=u(x,t):\Omega_{T}\to\R$,
\begin{equation}\label{i.1}
\begin{cases}
\mathbf{D}_{t}u-\mathcal{L}_{1}u-\mathcal{K}*\mathcal{L}_{2}u+f(u)=0\quad\text{in}\quad
\Omega_{T},\\
u(x,0)=u_{0}(x)\qquad\qquad\qquad\text{in}\qquad \bar{\Omega},
\end{cases}
\end{equation}
subject to  either the homogenous  Dirichlet boundary condition
\begin{equation}\label{i.2}
u=0\quad\text{on}\quad \partial\Omega_{T},
\end{equation}
or the homogenous Neumann boundary condition
\begin{equation}\label{i.3}
\frac{\partial u}{\partial \mathbf{N}}=0 \quad\text{on}\quad
\partial\Omega_{T},
\end{equation}
where $\mathbf{N}=\{N_{1},...,N_{2}\}$ is the outward normal to
$\Omega$.

\noindent Here the symbol $*$ stands for  the usual time-convolution
product on $(0,t)$,
\[
(\mathfrak{h}_{1}*\mathfrak{h}_{2})(t)=\int\limits_{0}^{t}
\mathfrak{h}_{1}(t-s)\mathfrak{h}_{2}(s)ds.
\]
The fractional differential operator $\mathbf{D}_{t}$ is presented
both in form \eqref{1.6} and  more complex form than \eqref{1.6},
and reads as
\begin{equation}\label{i.4}
\mathbf{D}_{t}u=\begin{cases} \rho\mathbf{D}_{t}^{\nu}
u+\rho_i\sum_{i=1}^{M}\mathbf{D}_{t}^{\nu_{i}}
u,\qquad\,\text{if}\quad
\mathbf{D}_{t}\quad \text{is  the I type \textbf{FDO}},\\
\\
\mathbf{D}_{t}^{\nu}(\rho
u)+\sum_{i=1}^{M}\mathbf{D}_{t}^{\nu_{i}}(\rho_i
u),\quad\text{if}\quad \mathbf{D}_{t}\quad \text{is  the II type
\textbf{FDO}}
\end{cases}
\end{equation}
 with any fixed $\nu\in(0,1)$ and $\nu_{i}\in(0,\nu)$,
and with given positive functions $\rho=\rho(t),$
$\rho_{i}=\rho_{i}(t)$.

Coming to the remaining  operators, $\mathcal{L}_{1}$ is a linear
elliptic operator of the second order with time-dependent
coefficients, while the operator $\mathcal{L}_{2}$ is a first-order
linear differential operator. Their precise forms will be given in
Sections \ref{s3}, where we detail the main assumptions in the
model.

 Published works concerning with the multi-term  fractional
diffusion/wave equations, i.e. the  equation with  the operator
\eqref{1.6} are quite limited in spite of rich literature on their
single-term version.
 Exact solutions of linear multi-term fractional diffusion
equations with $q_i$ being positive constants on bounded domains are
searched employing eigenfunction expansions in \cite{DB,MGSKA}.
Solvability along with a maximum principle and the longtime
asymptotic
 behavior of the solution for the initial-boundary value problems to  the linear subdiffusion equation with \eqref{1.6} are studied in \cite{LHY,LY,LSY}
 (see also references therein). Finally, we quote \cite{KTT},
 where existence and nonexistence of the mild solutions to the Cauchy problem for semilinear subdiffusion equation  with the operator
 \eqref{1.6} are discussed. In particular, the authors obtain the
 Fujita-type and Escobedo-Herrero-type critical exponents for this
 equation and the system.

Coming to the equation in \eqref{i.1} with the operator
$\mathbf{D}_{t}$ given via  \eqref{i.4} and related problems, we
point out two principal differences with respect to the
aforementioned papers. In the case of $\mathbf{D}_{t}$ being the II
type \textbf{FDO}, the first embarrassment deals with the
calculating Caputo fractional derivatives of a product of two
functions: the desired solution $u$ and the prescribed coefficients
$\rho,\rho_{i}$. Incidentally, we recall that the well-known Leibniz
rule does not work in the case of fractional derivatives. Moreover,
the equation in \eqref{i.1} models nonlinear processes by means
$f(u)$ and has additional nonlocal effects via memory terms,
$\mathcal{K}*\mathcal{L}_{2}u$.

The linear version of \eqref{i.1} with \eqref{i.4} subject to
various type boundary conditions with the coefficients in
$\mathbf{D}_{t}$ being alternating sign is discussed in
\cite{PSV,V}. For any fixed time $T,$ the existence and uniqueness
of a solution to semilinear problem \eqref{i.1}, \eqref{i.4} with
the \textbf{DBC},  the \textbf{NBC} or the boundary condition of the
third kind
 are analyzed in \cite{SV,V,V1}. Namely, the well-posedness of these problems in the fractional
H\"{o}lder  and Sobolev spaces is established   in the
one-dimensional case in \cite{SV}. As for the multidimensional case,
the one-valued classical or strong solvability to these problems are
discussed in \cite{V,V1} for each fixed $T>0$.

Nevertheless, a longtime behavior of these solutions  (i.e. as
$t\to+\infty$) to semilinear initial-boundary value problems similar
to \eqref{i.1}-\eqref{i.4} is still open issue. Thus, making proper
assumptions on the coefficients, the given functions and the
nonlinear term   $f$, we aim to fill this gap, providing the
estimates of $\|u(\cdot,t)\|_{W^{1,2}(\Omega)}$ and
$(|\mathcal{K}|*\|u\|_{W^{1,2}(\Omega)})(t)$ for all $t\geq 0$. In
particular, in the case of constant coefficients in the operator
$\mathbf{D}_{t}$ and the absence of the memory terms, we obtain the
estimate
\[
\|u(\cdot,t)\|_{W^{1,2}(\Omega)}\leq
C_{0}+\|u_{0}\|_{W^{1,2}(\Omega)}\mathfrak{g}(t),
\]
 where $C_{0}$ is a positive constant depending only on the structural parameters related with the nonlinear term $f(u)$ . The positive function $\mathfrak{g}(t)$ (defined with multinomial Mittag-Leffler functions
  depending only on the orders $\nu_{i}$, $\nu$ and the coefficients of $\mathbf{D}_{t}$ and $\mathcal{L}_{2}$)  has a decay at infinity of
 $t^{-\nu_{M}}$.
 In  more general case, we achieve
the following bound of a solution to \eqref{i.1}-\eqref{i.4}
\[
\|u\|_{L_{\infty}([0,+\infty),W^{1,2}(\Omega))}\leq
C[1+\|u_{0}\|_{W^{1,2}(\Omega)}].
\]
All these estimates in turn provide the existence of  absorbing sets
in the corresponding functional spaces.


\subsection*{Outline of the paper} The paper is organized as follows: in Section \ref{s2},
we introduce the notations and the functional spaces. The longtime
behavior of a solution (Theorems \ref{t3.1}-\ref{t3.3}), along
with the general hypothesis, is stated in Section \ref{s3}.
Theorems \ref{t3.1} and \ref{t3.0*} concern with the
one-dimensional case,
 while Theorems \ref{t3.2} and \ref{t3.3} describe a
behavior of a solution in the multidimensional case. The proof of
these results are carried out in Section \ref{s5} ($n=1$) and
Section \ref{s6} ($n\geq 2$). Section \ref{s4} is auxiliary  and
contains some technical and preliminary results from fractional
calculus, playing a key role in the course of the investigation.




\section{Functional Spaces and Notation}
\label{s2} \noindent Throughout this art, the symbol $C$ will
denote a generic positive constant, depending only on the
structural quantities of the problem.

In the course of our study, we will exploit
 the fractional H\"{o}lder and Sobolev spaces. To
this end, in what follows we take two arbitrary (but fixed)
parameters
\[
\alpha\in(0,1) \quad\text{and}\quad \nu\in(0,1).
\]
For any non-negative integer $l$,  any $p\geq 1,$ $s\geq 0$, and
any Banach space $(\mathbf{X},\|\cdot\|_{\mathbf{X}}),$  we
consider the usual spaces
\[
\C^{l+\alpha}(\bar{\Omega}),\quad W^{s,p}(\Omega),\quad
L_{p}(\Omega),\quad \C^{s}([0,T],\mathbf{X}),\quad
W^{s,p}((0,T),\mathbf{X}),\quad
L_{\infty}([0,+\infty),\mathbf{X}).
\]
Recall that for non-integer $s$ the space $W^{s,p}$ is called
Sobolev-Slobodeckii space (for its definition and properties see,
e.g.,   \cite[Chapter 1]{AF}, \cite[Chapter 1]{Gr}). Moreover, we
will use the notation $H^{-1}$ to the dual space of
$\overset{0}{W}\,^{1,2}(\Omega)$.

Denoting for $\beta\in(0,1)$
\begin{align*}
\langle v\rangle^{(\beta)}_{x,\Omega_T}&=\sup\Big\{\frac{|v(x_1,t)-v(x_2,t)|}{|x_1-x_2|^{\beta}}:\quad x_{2}\neq x_{1},\quad x_1,x_2\in\bar{\Omega}, \quad t\in[0,T]\Big\},\\
\langle v\rangle^{(\beta)}_{t,\Omega_T}&=\sup\Big\{\frac{|v(x,t_1)-v(x,t_2)|}{|t_1-t_2|^{\beta}}:\quad t_{2}\neq t_{1},\quad x\in\bar{\Omega}, \quad t_1,t_2\in[0,T]\Big\}.
\end{align*}
Then, we assert the following definition.
\begin{definition}\label{d2.1}
 A function $v=v(x,t)$ belongs to the class $\C^{l+\alpha,\frac{l+\alpha}{2}\nu}(\bar{\Omega}_{T})$, for $l=0,1,2,$ if the function $v$ and its corresponding derivatives are continuous and the norms here below are finite:
\begin{equation*}
\|v\|_{\C^{l+\alpha,\frac{l+\alpha}{2}\nu}(\bar{\Omega}_{T})}=
\begin{cases}
\|v\|_{\C([0,T],
\C^{l+\alpha}(\bar{\Omega}))}+\sum_{|j|=0}^{l}\langle
D_{x}^{j}v\rangle^{(\frac{l+\alpha-|j|}{2}\nu)}_{t,{\Omega}_{T}},
\qquad\qquad\qquad\qquad\quad\,
l=0,1,\\
\\
\|v\|_{\C([0,T],
\C^{2+\alpha}(\bar{\Omega}))}+\|\mathbf{D}_{t}^{\nu}v\|_{\C^{\alpha,\frac{\alpha}{2}\nu}
(\bar{\Omega}_{T})}+\sum_{|j|=1}^{2}\langle
D_{x}^{j}v\rangle^{(\frac{2+\alpha-|j|}{2}\nu)}_{t,{\Omega}_{T}},\quad l=2.
\end{cases}
\end{equation*}
\end{definition}
\noindent In a similar way, for $l=0,1,2,$ we introduce the space $\C^{l+\alpha,\frac{l+\alpha}{2}\nu}(\partial\Omega_{T})$.

The properties of these spaces have been discussed in
\cite[Section 2]{KPV3}. As for the limiting case $\nu=1$, these
classes boil down to the usual parabolic H\"{o}lder spaces
 (see
e.g.\cite[(1.10)-(1.12)]{LSU}).

Throughout this art, for any $\theta>0$, we use the notation
\[
\omega_{\theta}=\frac{t^{\theta-1}}{\Gamma(\theta)}
\]
and define the fractional Riemann-Liouville integral and the
derivative of order $\theta$, respectively, of a function
$v=v(x,t)$ with respect to time $t$ as
\[
I_{t}^{\theta}v(x,t)=(\omega_{\theta}*v)(x,t)\qquad
\partial_{t}^{\theta}v(x,t)=\frac{\partial^{\lceil\theta\rceil}}{\partial
t^{\lceil\theta\rceil}}(\omega_{\lceil\theta\rceil-\theta}*v)(x,t),
\]
where $\lceil\theta\rceil$ is the ceiling function of $\theta$
(i.e. the smallest integer greater than or equal to $\theta$).

It is apparent that, for $\theta\in(0,1)$ there holds
\[
\partial_{t}^{\theta}v(x,t)=\frac{\partial}{\partial
t}(\omega_{1-\theta}*v)(x,t).
\]
Accordingly, the Caputo fractional derivative of the order
$\theta\in(0,1)$ to the function $v(x,t)$ can be represented as
\begin{align}\label{2.2}\notag
\mathbf{D}_{t}^{\theta}v(\cdot,t)&=\frac{\partial}{\partial
t}(\omega_{1-\theta}*v)(\cdot,t)-\omega_{1-\theta}(t)v(\cdot,0)\\
&
=
\partial_{t}^{\theta}v(\cdot,t)-\omega_{1-\theta}(t)v(\cdot,0)
\end{align}
provided that both derivatives exist.

Finally, setting
\[
\bar{\beta}=\{\beta_1,\beta_2,...,\beta_{m}\},
\]
 we will notice
 the multinomial Mittag-Leffler function
introduced in \cite{HL} (see (3.16) therein) and originally named
as "multivariate" Mittag-Leffler function:
\[
E_{\bar{\beta},\beta_{0}}(\bar{z})=
E_{(\beta_{1},...,\beta_{m})\beta_{0}}(z_{1},...,z_{m})=\sum_{k=0}^{+\infty}\sum_{k_1+...+k_{m}=k}
\frac{k!}{k_{1}!...k_{m}!}\frac{\prod_{i=1}^{m}z_{j}^{k_{j}}}{\Gamma(\beta_{0}+\sum_{j=1}^{m}\beta_{j}k_{j})},
\]
where $z_{j}\in\mathbb{C}$, $\beta_{j}>0,$ $\beta_{0}\in\R$.

It is worth noting that these functions are utilized for solving
multi-term fractional differential equations with constant
coefficients by the operational approach (see e.g. \cite{LG,L}).

For $0<\beta_{m}<\beta_{m-1}<...<\beta_{1}\leq\beta_{0}$ and
$\bar{d}=\{d_1,d_2,...,d_{m}\}$, $d_{i}\in\R$, we denote
\begin{equation}\label{4.1}
\mathcal{E}_{\bar{\beta},\beta_{0}}(t;\bar{d})=t^{\beta_{0}-1}E_{(\beta_{1},...,\beta_{m})\beta_{0}}(-d_{1}t^{\beta_1},...,-d_{m}t^{\beta_{m}}).
\end{equation}



\section{Main Results}
\label{s3}

\noindent
 First, we state our general hypothesis on the structural
terms of the model.
\begin{description}
 \item[h1 (Conditions on the fractional order of the derivatives)] We assume that
\begin{equation*}
 0<\nu_1<\nu_2<...<\nu_{M}<\nu<1.
\end{equation*}

\smallskip
    \item[h2 (Conditions on the operators)]
The operators appearing in the equation in \eqref{i.1} read
\begin{align*}
\mathcal{L}_{1}&=
\begin{cases}
\frac{\partial}{\partial x}\big(a_{1}(x,t)\frac{\partial}{\partial
x}\big)+a_{0}(x,t),\quad n=1,\\
\\
\sum_{i=1}^{n} \frac{\partial^{2}}{\partial
x_{i}^{2}}+a_{0}(x,t),\qquad\quad n\geq 2,
\end{cases}
\\
\mathcal{L}_{2}&=\sum_{i=1}^{n}b_{i}(x,t)\frac{\partial}{\partial
x_i}+b_{0}(x,t).
\end{align*}
     There are positive constants $\delta$ and
         $\delta_{i}$ $i=0,1$ such that
     \begin{align*}
-a_0(x,t)&\geq\delta_0>0,\quad a_{1}(x,t)\geq \delta_{1}>0,\\
 \rho(t)&\geq \delta>0,\quad \rho_{k}(t)\geq 0\quad
k=1,2,...,M,
    \end{align*}
   for any $x\in\bar{\Omega}$ and $t\geq 0$.
       \smallskip
    \item[h3 (Regularity of the coefficients)]
    We require
  \[
      b_{j}, a_{0},a_{1}, \frac{\partial a_{1}}{\partial x}, \frac{\partial a_{0}}{\partial x_{i}}\in
   L_{\infty}([0,+\infty),\C(\bar{\Omega})),\quad
   j=0,1,..., n,\quad i=1,..,n,\]
   \begin{equation}\label{3.0}
        \rho,\rho_k,\frac{d \rho}{d t},\frac{d \rho_{k}}{d t}\in L_{\infty}([0,+\infty))\quad \text{and}\quad \frac{d \rho}{d t},\frac{d
        \rho_{k}}{d
        t}\geq 0\quad \text{for all}\quad t\geq 0, \quad
        k=1,2,...,M.
\end{equation}

\noindent Moreover, there exist nonnegative $T_{0}, T_{k}$,
$k=1,2,...,M,$ such that
\[
\rho(t)=\rho(T_0)\quad \text{for all}\quad t\geq T_{0},\quad
\rho_{k}(t)=\rho(T_k)\quad \text{for all}\quad t\geq T_{k}.
\]

    \smallskip
        \item[h4 (Condition on the given functions)]
        \[
u_0\in
\begin{cases}
\overset{0}{W}\,^{1,2}(\Omega)\quad \text{in the \bf{DBC} case},\\
W^{1,2}(\Omega)\quad \text{in the \bf{NBC} case},
\end{cases}
        \]
$$
\mathcal{K}(t)\in L_{1}(\R^{+})
$$
and for some $\nu^{*}\in[0,\nu]$ the kernel $\mathcal{K}$
satisfies the bound
 \[
|\mathcal{K}|\leq C t^{-\nu^{*}}
\]
for all $t\geq 0$.
    \smallskip

    \item[h5 (Conditions on the nonlinearity)]
    We assume that $f\in\C^{1}(\R)$, and for some nonnegative
 constants $L_{i}$, i=1,2,3,4, and $\gamma\geq 1,$ the
 inequalities hold
        \[
        \begin{cases}
               |f(u)|\leq L_1(1+|u|^{\gamma}),\\
        uf(u)\geq -L_2+L_3|u|^{\gamma+1},\\
        f'(u)\geq -L_{4}.
        \end{cases}
        \]
                       \end{description}

    We notice that the assumptions on the nonlinear term tell us that $f(u)$ is  a polynomial of odd degree with the positive
leading coefficient (see \cite{GPM}).

Our first main result is related to a behavior of a solution to
\eqref{i.1}-\eqref{i.3} for any $t>0$, which, in some particular
cases  (see below), can be reformulated as an existence of absorbing
sets.
\begin{theorem}\label{t3.1}
Let $n=1$ and conditions h1-h5 hold. If $\frac{\partial
a_{1}}{\partial x}\not\equiv 0$, then we additionally assume that
there is a positive value $\delta^{*}\in(0,\delta_{1})$ such that
\begin{equation}\label{3.1}
\delta^{*}(\delta_1+\delta_0)-\bigg\|\frac{\partial
a_{1}}{\partial
x}\bigg\|^{2}_{L_{\infty}([0,+\infty),\C(\bar{\Omega}))}>0.
\end{equation}
Then a solution  of \eqref{i.1}-\eqref{i.3} satisfies the estimates
\begin{equation}\label{3.2}
\|u\|_{L_{\infty}([0,+\infty),W^{1,2}(\bar{\Omega}))} \leq
C_{0}[1+\|u_{0}\|_{W^{1,2}(\bar{\Omega})}],
\end{equation}
\begin{equation}\label{3.2*}
\|u\|_{L_{\infty}([0,+\infty),\C(\bar{\Omega}))}+\|\mathcal{K}*u\|_{L_{\infty}([0,+\infty),W^{1,2}(\bar{\Omega}))}\leq
C_{1}[1+\|u_{0}\|_{W^{1,2}(\bar{\Omega})}].
\end{equation}
Here the positive constants $C_{0}$ and $C_{1}$ depend only on the
structural parameters in the model, the Lebesgue measure of
$\Omega$, and the corresponding norms of the coefficients and of the
memory kernel.
\end{theorem}
\begin{example}
The following are examples of the coefficients $a_{0}$ and $a_{1}$
satisfying condition \eqref{3.1}:
\[
\Omega=(1,2),\quad a_{1}(x,t)=x,\quad a_{0}(x,t)=-3-e^{-t}.
\]
It is apparent that in this case
\[
\delta_{1}=1, \quad \delta_{0}=3
\]
and hence, $\eqref{3.1}$ is fulfilled with any
$\delta^{*}\in(1/4,1)$.
\end{example}

Coming to the multidimensional case, we claim the following.
\begin{theorem}\label{t3.2}
Let $n\geq 2$. Under assumptions h1-h5, a solution of
\eqref{i.1}-\eqref{i.3} satisfies the bound \eqref{3.2} and
\begin{equation*}
\|\mathcal{K}*u\|_{L_{\infty}([0,+\infty),W^{1,2}(\bar{\Omega}))}\leq
C_{1}[1+\|u_{0}\|_{W^{1,2}(\bar{\Omega})}].
\end{equation*}
\end{theorem}
\noindent   In the case of the constant coefficients in the operator
$\mathbf{D}_{t}$ and $\mathcal{K}\equiv 0$, the proof of Theorems
\ref{t3.1} and \ref{t3.2} provide more accurate estimates of
solution $u.$  Indeed, setting
\begin{align*}
\bar{\beta}&=(\nu,\nu-\nu_1,\nu-\nu_2,...,\nu-\nu_{M}),\\
\mathcal{E}_{\bar{\beta},\nu}&=t^{\nu-1}
E_{\bar{\beta},\nu}\bigg(-t^{\nu}\frac{C_{0}^{*}}{\rho},-t^{\nu-\nu_{1}}\frac{\rho_{1}}{\rho},...,
-t^{\nu-\nu_{M}}\frac{\rho_{M}}{\rho}\bigg),
\end{align*}
where the positive value $C_{0}^{*}$ depends only on the
corresponding norms of the coefficients in $\mathcal{L}_{1}$, we
claim the following.
\begin{corollary}\label{c1}
Let assumptions of Theorem \ref{t3.1} or \ref{t3.2} hold,
$\mathcal{K}\equiv 0$ and let  $\rho_{i},$ $\rho$ be constant.  Then
the estimates hold
\begin{align*}
\|u(\cdot, t)\|_{\C(\bar{\Omega})}&\leq
C_{0}[1+\|u_{0}\|_{W^{1,2}(\bar{\Omega})}\mathfrak{g}(t)],\quad\text{if
}\quad n=1;\\
\|u(\cdot, t)\|_{W^{1,2}(\Omega)}&\leq
C_{0}[1+\|u_{0}\|_{W^{1,2}(\bar{\Omega})}\mathfrak{g}(t)],\quad\text{if
}\quad n\geq 2
\end{align*}
for any $t\geq 0$.

\noindent Here the positive function
\[
\mathfrak{g}(t)=\bigg|1-\frac{C_{0}^{*}}{\rho}\mathcal{E}_{\bar{\beta},\nu}*1\bigg|+
\sum_{i=1}^{M}\frac{\rho_{i}}{\rho}(\mathcal{E}_{\bar{\beta},\nu}*\omega_{1-\nu_i})(t)+
(\mathcal{E}_{\bar{\beta},\nu}*\omega_{1-\nu})(t)
\]
 vanishes for
$t\to+\infty$ and has the  asymptotic behavior
\[
\mathfrak{g}\sim
\begin{cases}
\frac{\rho}{C_{0}^{*}}\big[
\sum_{i=1}^{M}\frac{\rho_{i}}{\rho}t^{-\nu_{i}}+ t^{-\nu}\big],\quad
\text{if}\quad t\to+\infty,\\
2,\qquad\qquad\qquad \qquad\qquad\quad\text{if}\quad t\to+0.
\end{cases}
\]
\end{corollary}

The results of Theorems \ref{t3.1} and \ref{t3.2} under additional
assumptions on the coefficients are true if h5 takes place with
$\gamma\in(0,1)$. Namely, introducing the value
\[
\delta_{1}^{*}=\frac{\delta^{*}(\delta_1+\delta_0)-\|\partial
a_{1}/\partial x\|^{2}_{L_{\infty}([0,+\infty),\C(\bar{\Omega}))}}
{\delta^{*}},
\]
where $\delta^{*}\in(0,\delta_{1})$ is defined in Theorem
\ref{t3.1}, we have.
\begin{theorem}\label{t3.0*}
Let $n=1$ and assumptions h1-h4 and \eqref{3.1} hold. Moreover, we
require that $f(u)$ satisfies  h5  with $\gamma\in(0,1)$ and there
is a positive
\[
\delta_{2}^{*}\in
\begin{cases}
(0,\delta_{1}-\delta^{*})\qquad \text{if}\quad
\frac{\partial a_{1}}{\partial x}\not\equiv 0,\\
(0,\delta_{1})\qquad\qquad \text{if }\quad\frac{\partial
a_{1}}{\partial x}\equiv 0,
\end{cases}
\]
such that the inequality
\begin{equation}\label{3.4}
\delta_{0}-\frac{L_{4}^{2}}{\delta_{2}^{*}}-\frac {\|\partial
a_{0}/\partial
x\|^{2}_{L_{\infty}([0,+\infty),\C(\bar{\Omega}))}}{\delta_{1}^{*}}>0
\end{equation}
is fulfilled. Then the results of Theorem \ref{t3.1}  hold.
\end{theorem}
\begin{theorem}\label{t3.3}
Let $n\geq 2$, and assumptions h1-h4 hold. We assume that $f(u)$
satisfies h5 with $\gamma\in(0,1)$ and there are positive
quantities  $\delta_{3}^{*}\in(0,1+\delta_0)$,
$\delta_{4}^{*}\in(0,1)$ such that
\begin{equation}\label{3.4*}
\delta_{0}-\frac{L_{4}^{2}}{\delta_{4}^{*}}-\frac
{\|D_{x}a_{0}\|^{2}_{L_{\infty}([0,+\infty),\C(\bar{\Omega}))}}{\delta_{3}^{*}}>0.
\end{equation}
Then the results of Theorem \ref{t3.2} hold.
\end{theorem}
\begin{example}
It is apparent that the following $a_{0},$ $a_{1}$, $f(u)$ meet the
requirements of Theorems \ref{t3.0*} and \ref{t3.3}:
\[
a_{1}(x,t)=1,\quad a_{0}(x,t)=-3-e^{-t},\quad
f(u)=u(1+u^{2})^{a-b},\quad\text{where}\quad \max\{0,b-1/2\}<a<b<1.
\]
Indeed, in this case, the straightforward calculations arrive at the
relations:
\[
\delta_{0}=3,\quad\delta_{1}=1,\quad L_{1}=L_{2}=L_{3}=1,\quad
L_{4}=0,\quad \gamma=1-2b+2a.
\]
\end{example}

In  these  theorems, we do not specify the type of a solution to
\eqref{i.1}-\eqref{i.3}.  In fact, the definitions of a weak/strong
and  classical solution given here below tell us that all these
solutions satisfy the bounds stated in Theorems
\ref{t3.1}-\ref{t3.3} and Corollary \ref{c1}.

\begin{definition}\label{d3.1} (\textit{Strong solution}) A
function $u$ is called a strong solution to the initial value
problem \eqref{i.1} subjected to either Dirichlet boundary
condition \eqref{i.2} or the Neumann  boundary condition
\eqref{i.3} if

\noindent$\bullet$ $u\in\C(\bar{\Omega}_{T})\cap
L_{2}((0,T),W^{2,2}(\Omega)),$ $\mathbf{D}^{\nu}_{t}u,$
$\mathbf{D}^{\nu_{i}}_{t}u\in L_{2}(\Omega_{T})$, $i=1,...,M;$

\noindent$\bullet$ the initial condition in \eqref{i.1} and the
boundary condition \eqref{i.2} or \eqref{i.3} hold;

\noindent$\bullet$ for $T>0$ and any $\Phi\in L_{2}(\Omega_{T})$
the identity
\[
\int_{\Omega}\int_{0}^{T}(\mathbf{D}_{t}u-\mathcal{L}_{1}u-\mathcal{K}*\mathcal{L}_{2}u+f(u))\Phi
dxdt=0
\]
is performed.
\end{definition}

It is apparent that a function
$u\in\C^{2+\alpha,\frac{2+\alpha}{2}\nu}(\bar{\Omega}_{T})$ having
additional regularity $\mathbf{D}_{t}^{\nu_{i}}u\in
\C^{\alpha,\frac{\nu\alpha}{2}}(\bar{\Omega}_{T}),$ $i=1,...,M,$ is
called a classical solution if relations \eqref{i.1}-\eqref{i.3} are
fulfilled  in the classical sense.

As for the definition of a weak solution to \eqref{i.1}-\eqref{i.3},
here we restrict ourselves to the case of \textbf{DBC}. In  the
\textbf{NBC} case, the definition is given in the similar way.
\begin{definition}\label{d3.2} (\textit{Weak solution}) We say
that $u$  is a weak solution of \eqref{i.1} and \eqref{i.2} if

\noindent$\bullet$ $u\in L_{2}(0,T,
\overset{0}{W}\,^{1,2}(\Omega)),$ $\omega_{1-\nu}*(\rho
u-\rho(0)u_0),$ $\omega_{1-\nu_{i}}*(\rho_{i} u-\rho_{i}(0)u_0)\in
\overset{0}{W}\,^{1,2}((0,T),H^{-1}(\Omega));$

\noindent$\bullet$ for any test function $\Phi\in
W^{1,2}((0,T),L_{2}(\Omega))\cap L_{2}((0,T),\overset{0}{W}\,
^{1,2}(\Omega)),$ $\Phi|_{t=T}=0,$ there holds
\begin{align*}
&-\int_{0}^{T}\int_{\Omega}\Phi_{t}\bigg\{ \omega_{1-\nu}*(\rho
u-\rho(0)u_0)+\sum_{i=1}^{M} \omega_{1-\nu_{i}}*(\rho_{i}
u-\rho_{i}(0)u_0)\bigg\}dxdt\\
&+ \int_{0}^{T}\int_{\Omega}\bigg[\sum_{i=1}^{n}\frac{\partial
u}{\partial x_{i}}\frac{\partial \Phi}{\partial
x_{i}}-a_{0}u\Phi\bigg]dxdt
-\int_{0}^{T}\int_{\Omega}(\mathcal{K}*\mathcal{L}_{2}u)\Phi dxdt
\\
& = -\int_{0}^{T}\int_{\Omega}f(u)\Phi dx dt.
\end{align*}
\end{definition}

We finally observe that the smoothness of the given functions
required in Theorems \ref{t3.1}-\ref{t3.3} as well as the
restriction on  $\nu_{i}$, $i=1,...,M,$ do not provide, generally
speaking, the classical/strong solvability to problems like
\eqref{i.1}-\eqref{i.3}. We refer to \cite[Theorems 4.1 and
4.4]{SV}, \cite[Theorem 2]{V} and \cite[Theorems 3.1 and 3.2]{V1},
where under stronger conditions on the given data, the existence and
uniqueness of classical and strong solutions to more general
problems than \eqref{i.1}-\eqref{i.3}  are established for each
fixed positive $T$. For example, if
\[
\nu_{i}\in\bigg(0,\frac{\nu(2-\alpha)}{2}\bigg)\quad\text{and}\quad
u_{0}\in
\begin{cases}
\overset{0}{W}\,^{1,2}(\Omega)\cap W^{2,2}(\Omega)\quad\text{in the \textbf{DBC} case,}\\
W^{2,2}(\Omega)\qquad\qquad\qquad\text{in the \textbf{NBC} case,}
\end{cases}
\]
then \cite[Theorem 4.4]{SV} asserts the  one-valued strong
solvability to \eqref{i.1}-\eqref{i.3} for each fixed $T$ in the
one-dimensional case,  i.e. $u\in L_{2}((0,T),W^{2,2}(\Omega))\cap
W^{\nu,2}((0,T),L_{2}(\Omega))$.

\noindent Nevertheless, assumptions in this art are sufficient to
obtain the asymptotic behavior of $u$  for large time (including
$t\to +\infty$).

If the coefficients of the operators in  \eqref{i.1} are
time-independent, then Theorems \ref{t3.1}-\ref{t3.3} ensure the
existence of the absorbing set $\mathcal{B}_{0}\subset\mathbb{H}$,
where $\mathbb{H}$ is a Hilbert space depending on the smoothness of
the initial data $u_{0}$.

Indeed, appealing to the definition in \cite[Section 3]{GPM}, we
denote
 an open ball centered at
$0$ and a radius $R>0$ in a Hilbert space $\mathbb{H}$ by
$\mathcal{B}_{0}=\mathcal{B}_{R}(0)$.
\begin{definition}\label{d3.3}
 A
bounded set $\mathcal{B}_{0}\subset \mathbb{H}$ is called an
absorbing set to \eqref{i.1}-\eqref{i.3}  if for any initial value
$u_{0}\in \mathcal{B}_{R}(0)\subset\mathbb{H}$ there exists
$t_{\mathbb{H}}=t_{\mathbb{H}}(R)$ such that the solution
$u(\cdot,t)\in \mathcal{B}_{0}$ for any $t\geq t_{\mathbb{H}}$.
\end{definition}
Bearing in mind this definition, we reformulate Theorems
\ref{t3.1}-\ref{t3.3}.
\begin{theorem} Let
the  coefficients in the operators of the equation in \eqref{i.1} be
time-independent. Under assumptions of Theorems
\ref{t3.1}-\ref{t3.3}, there exists the absorbing set in
\[
\mathbb{H}=\begin{cases}
\overset{0}{W}\,^{1,2}(\Omega)\quad\text{in the \textbf{DBC} case,}\\
W^{1,2}(\Omega)\quad\text{in the \textbf{NBC} case,}
\end{cases}
\]
for problem \eqref{i.1}-\eqref{i.3}.
\end{theorem}
\begin{remark}\label{r3.1}
If $\rho_{i}\equiv 0,$ $i=1,...,M,$ then the coefficient $\rho$ may
be dependent on not only time but also the space variables too.
Indeed, Theorems \ref{t3.1}-\ref{t3.3} hold if in \eqref{3.0} we
will require
\begin{align*}
&\rho=\rho(x,t)\in L_{\infty}([0,+\infty),\C(\bar{\Omega}))\quad
\text{and} \quad \frac{\partial \rho}{\partial t}(x,t)\in
L_{\infty}([0,+\infty),\C(\bar{\Omega})),\\
& \frac{\partial \rho}{\partial t}(x,t)\geq 0\quad\text{for
all}\quad (x,t)\in \bar{\Omega}_{\infty},\quad
\rho(x,t)=\rho(x,T_{0})\quad \text{for all}\quad t\geq T_{0}\quad
\text{and}\quad x\in\bar{\Omega}.
\end{align*}
\end{remark}
\begin{remark}\label{r3.3}
We notice that our requirements on the kernel $\mathcal{K}$ admit
the case $\mathcal{K}\equiv 0,$ telling us that the multi-term
semilinear equation
\[
\mathbf{D}_{t}u-\mathcal{L}_{1}u+f(u)=0
\]
fits in our analysis and is discussed by the theorems above.
\end{remark}

It is worth noting that the proof of Theorems \ref{t3.0*} and
\ref{t3.3} is similar to the arguments leading to Theorems
\ref{t3.1} and \ref{t3.2}, respectively. Thus, in this art (see
Sections \ref{s5}-\ref{s6}), we prove Theorems \ref{t3.1} and
\ref{t3.2}, while the verification of Theorems \ref{t3.0*} and
\ref{t3.3} is left to the interesting readers. Besides, in our
analysis we focus on $\mathbf{D}_{t}$ being the II type
\textbf{FDO}, sine the case the I type \textbf{FDO} is simpler and
tackled with the similar arguments.



\section{Technical Results}
\label{s4}

\noindent In this section we collect some useful properties of
fractional derivatives and integrals, as well as several
preliminaries results that will be significant in our investigation.
First, we subsume \cite[Lemma 2.1]{KST}, \cite[Proposition
4.1]{KPV2}, \cite[Proposition 4.2 ]{KPV3}, \cite[Proposition 1]{KV}
as the following claim.
\begin{proposition}\label{pA.0}
The following hold.

\noindent (i) For any given positive numbers $\theta_1$ and
$\theta_2$ and a summable kernel $k=k(t)$, there are relations
\[
(\omega_{\theta_1}*\omega_{\theta_2})(t)=\omega_{\theta_1+\theta_{2}}(t),\,
(1*\omega_{\theta_{1}})(t)=\omega_{1+\theta_{1}}(t),\,
\omega_{\theta_{1}}(t)\geq CT^{\theta_{1}-1},\,
(\omega_{\theta_{1}}*k)(t)\leq C\omega_{\theta_{1}}(t).
\]
Here the positive constant $C$ depends only on $T,$ $\theta_{1}$
and $\|k\|_{L_{1}(0,T)}$.

\noindent (ii) Let $k\in L_{1}(0,T)$, $\theta\in(0,1)$, and let
$v=v(t)$ be a bounded function on $[0,T]$. Then there are the
inequalities:
\begin{align*}
I_{t}^{\theta}(k*v)(t)&=(w*v)(t)\quad\text{with}\quad
w=I_{t}^{\theta}k(t),\\
\|I_{t}^{\theta}k\|_{L_{1}(0,T)}&\leq\frac{T^{\theta}}{\Gamma(1+\theta)}\|k\|_{L_{1}(0,T)}.
\end{align*}

\noindent (iii) Let $\theta\in(0,1),$ $\mathfrak{C}_{i},$
$i=0,1,2,$ be  nonnegative numbers,
$\mathfrak{C}_{1}+\mathfrak{C}_{2}>0,$ $\mathfrak{C}_{0}>0$.
Assume that a function $k_{1}\in L_{1}(0,T)$  and the relations
hold:
\[
I_{t}^{\theta}|k_{1}|(t)\in\C([0,T])\qquad\text{and}\qquad
I_{t}^{\theta}k_{1}(0)=0.
\]
If a function $v\in\C([0,T])$ satisfies the estimate
\[
|v|\leq \mathfrak{C}_{0}+\mathfrak{C}_{1}I_{t}^{\theta}|k_1
v|+\mathfrak{C}_{2}(|k_2|*|v|)(t)
\]
for all $t\in[0,T]$ and any $k_{2}\in L_{1}(0,T)$, then the bound
\[
|v|\leq \mathfrak{C}_{0}\mathfrak{C}
\]
 holds with the positive quantities depending only on
 $T,\mathfrak{C}_{i},$ $i=1,2,$ and the corresponding norms of
 $k_1$ and $k_{2}$.
 \end{proposition}

The next result  includes  \cite[Theorem 2.4]{SKM},
\cite[Proposition 5.1]{SV}, \cite[Theorem 2.3]{VZ}.
\begin{proposition}\label{pA.1}
The following holds.

\noindent(i) Let $v\in\C([0,T])$ then
\[
I_{t}^{\theta}\partial_{t}^{\theta}v=v.
\]

\noindent (ii) Let $\theta,\theta_1\in(0,1)$ and
$\theta_1>\theta/2,$ $v\in\C^{\theta_1}([0,T])$. For any even
integer $p\geq 2$ the inequalities are true
\[
\partial_{t}^{\theta}v^{p}(t)\leq
\partial_{t}^{\theta}v^{p}(t)+(p-1)v^{p}(t)\omega_{1-\theta}(t)\leq p v^{p-1}(t)\partial_{t}^{\theta}v(t).
\]
If $v$  is nonnegative, then these bounds hold for any integer odd
$p$.

\noindent Moreover, if $v\in L_{2}(0,T)$ and
$\omega_{1-\theta}*|v-v(0)|^{2}\in \overset{0}{W}\,^{1,2}(0,T)$,
then these inequalities hold for  a.e. $t\in[0,T]$.
\end{proposition}

At this point, for given functions $ w_1$ and $w_2$, we define
\begin{align*}
\mathcal{J}_{\theta}(t)&=\mathcal{J}_{\theta}(t;w_1,w_2)=\int_{0}^{t}\frac{[w_1(t)-w_1(s)]}{(t-s)^{1+\theta}}[w_2(s)-w_2(0)]ds,\\
\mathcal{W}({w_1})&=\mathcal{W}(w_1;t,\tau)=\int_{0}^{1}\frac{\partial
w_{1}}{\partial z}(z)ds,\quad\text{where}\\
z&=st+(1-s)\tau,\qquad 0<\tau<t<T,
\end{align*}
and assert the results related to the fractional differentiation
of the product.
\begin{proposition}\label{pA.2}
Let $\theta\in(0,1),$ $w_1\in\C^{1}([0,T]),$ $w_2\in\C([0,T])$.

\noindent(i) If $\mathbf{D}_{t}^{\theta}w_2$ belongs either to
$\C([0,T])$ or to $L_{p}(0,T)$, $p\geq 2,$ then there are
equalities:
\begin{align*}
\mathbf{D}_{t}^{\theta}(w_1w_2)&=w_1(t)\mathbf{D}_{t}^{\theta}w_2(t)+w_{2}(0)\mathbf{D}_{t}^{\theta}w_1(t)+\frac{\theta}{\Gamma(1-\theta)}
\mathcal{J}_{\theta}(t;w_{1},w_2),\\
\partial_{t}^{\theta}(w_1w_2)&=w_1(t)\mathbf{D}_{t}^{\theta}w_2(t)+w_{2}(0)\partial_{t}^{\theta}w_1(t)+\frac{\theta}{\Gamma(1-\theta)}\mathcal{J}_{\theta}(t;w_{1},w_2),
\end{align*}
and $\mathbf{D}_{t}^{\theta}(w_1w_2),$
$\partial_{t}^{\theta}(w_1w_2)$ have the regularity:
\[
\mathbf{D}_{t}^{\theta}(w_1w_2),\,\partial_{t}^{\theta}(w_1w_2)\in
\begin{cases}
\C([0,T]),\quad\text{if}\quad \mathbf{D}_{t}^{\theta}w_2\in
\C([0,T]),\\
L_{p}(0,T),\quad\text{if}\quad \mathbf{D}_{t}^{\theta}w_2\in
L_{p}(0,T).
\end{cases}
\]

 \noindent (ii) For any $\theta_{1}\geq\theta>0$ and each
$t\in[0,T]$, there hold
\begin{align*}
I_{t}^{\theta_1}(w_1\partial_{t}^{\theta}w_2)(t)&=I_{t}^{\theta_{1}-\theta}(w_1w_2)(t)-w_2(0)[I_{t}^{\theta_{1}-\theta}w_1-I_{t}^{\theta_{1}}(w_{1}\omega_{1-\theta})(t)]
- \theta I_{t}^{1+\theta_{1}-\theta}(\mathcal{W}(w_1)w_2)(t),\\
I_{t}^{\theta_1}(w_1\mathbf{D}_{t}^{\theta}w_2)(t)&=I_{t}^{\theta_{1}-\theta}(w_1w_2)(t)-w_2(0)I_{t}^{\theta_{1}-\theta}w_1
- \theta I_{t}^{1+\theta_{1}-\theta}(\mathcal{W}(w_1)w_2)(t).
\end{align*}

Besides
\begin{equation}\label{A.0}
I_{t}^{\theta}\mathcal{W}(w_1)(0)=0\qquad\text{and}\qquad
I_{t}^{\theta}\mathcal{W}(w_1)\in\C([0,T]).
\end{equation}
\end{proposition}
\begin{proof}
Equalities in (i) and (ii) (excepting relations in \eqref{A.0})
are stated in \cite[Proposition 5.5]{SV}. Thus,  here we are left
to verify \eqref{A.0}.  Bearing in mind the regularity of $w_1$,
we have
\[
0\leq |I_{t}^{\theta}\mathcal{W}(w_1)(t)|\leq
\|w_1\|_{\C^{1}([0,T])}I_{t}^{\theta}1\leq C
t^{\theta}\|w_1\|_{\C^{1}([0,T])}.
\]
Obviously, the last inequality leads to the equality in
\eqref{A.0}.

Coming to the regularity of $I_{t}^{\theta}\mathcal{W}(w_1)$, it
is a simple consequence of \cite[Lemma 2.8]{KST} and the
smoothness of $w_1$.
\end{proof}

\begin{corollary}\label{cA.3}
Let $\theta\in(0,1),$ and let $w_1,$ $w_2$ and $\frac{\partial
w_1}{\partial t}\in L_{\infty}([0,+\infty))$. We assume also that
$w_1$ is a non-negative non-decreasing function, and there is a
positive fixed $T_{w}>0$ such that
\[
w_1(t)=w(T_{w})\quad \text{for all}\quad t\geq T_{w}.
\]
Then, for each $t\geq 0,$ the function
$\mathcal{J}_{\theta}(t;w_1^{2},w_{2}^{2})$ satisfies the bound
\[
\frac{\mathcal{J}_{\theta}(t;w_1^{2},w_{2}^{2})}{2w_{1}(t)}\geq
-\frac{Cw_{2}^{2}(0)}{1-\theta}T_{w}^{1-\theta}\bigg\|\frac{d
w_1}{d t}\bigg\|_{L_{\infty}([0,+\infty))}
\]
with the positive $C$ being independent of $t$ and $T_{w}$.
\end{corollary}
\begin{proof}
Collecting the definition of
$\mathcal{J}_{\theta}(t;w_{1}^{2},w_{2}^{2})$ with nondecreasing
and nonnegativity of $w_1$ arrives at the estimate
\[
\frac{\mathcal{J}_{\theta}(t;w_{1}^{2},w_{2}^{2})}{2w_{1}(t)}\geq-
\frac{w_{2}^{2}(0)}{2w_{1}(t)}\int_{0}^{t}\frac{w_{1}^{2}(t)-w_{1}^{2}(s)}{(t-s)^{1+\theta}}ds
\geq
-w_{2}^{2}(0)\int_{0}^{t}\frac{w_{1}(t)-w_{1}(s)}{(t-s)^{1+\theta}}ds\equiv-w_{2}^{2}(0)i(t).
\]

\noindent In order to examine $i(t)$, we consider   two different
cases: $t\in[0,T_{w}]$ and $t>T_{w}$. If $t\leq T_{w}$, then bering
in mind the regularity of $w_{1}(t)$, we end up with the bound
\[
i(t)\leq C\frac{T_{w}^{1-\theta}}{1-\theta}\bigg\|\frac{d w_1}{d
t}\bigg\|_{L_{\infty}([0,+\infty))}
\]
with $C$ being independent of $t$ and  $T_{w}$.

In the case of $t>T_{w}$, we easily conclude that
\[
i(t)=\int_{0}^{T_{w}}\frac{w_{1}(t)-w_{1}(s)}{(t-s)^{1+\theta}}ds\leq
C\bigg\|\frac{d w_1}{d t}\bigg\|_{L_{\infty}([0,+\infty))}
\int_{0}^{T_{w}}\frac{ds}{(T_{w}-s)^{\theta}}
=C\frac{T_{w}^{1-\theta}}{1-\theta}\bigg\|\frac{d w_1}{d
t}\bigg\|_{L_{\infty}([0,+\infty))}.
\]
Finally, collecting all estimates arrives at the desired bound and
completes the proof of this claim.
\end{proof}
\begin{corollary}\label{cA.4}
Let $\theta\in(0,1)$, and let $w_1$ meet requirements of Corollary
\ref{cA.3}. Then the bound
\[
\frac{|\partial_{t}^{\theta} w_{1}^{2}|}{w_{1}(t)}\leq
\frac{w_{1}(0)}{t^{\theta}\Gamma(1-\theta)} + \frac{2\big\|\frac{d
w_1}{d
t}\big\|_{L_{\infty}([0,+\infty))}}{\Gamma(2-\theta)}T_{w}^{1-\theta}
\]
holds for each $t\geq 0$.
\end{corollary}
\begin{proof}
The straightforward calculations provide
\[
\frac{|\partial_{t}^{\theta} w_{1}^{2}|}{w_{1}(t)}\leq
\frac{w_{1}(0)}{t^{\theta}\Gamma(1-\theta)}
+\frac{2}{w_{1}(t)\Gamma(1-\theta)} \int_{0}^{t}w_{1}(\tau)\frac{d
w_1}{d \tau}(\tau)(t-\tau)^{-\theta}d\tau.
\]
Then, bearing in mind the smoothness of $w_1$, we deduce
\[
w_{1}^{-1}(t)\int_{0}^{t}w_{1}(\tau)\frac{d w_1}{d
\tau}(\tau)(t-\tau)^{-\theta}d\tau \leq
\frac{T_{w}^{1-\theta}}{1-\theta} \bigg\|\frac{d w_1}{d
t}\bigg\|_{L_{\infty}([0,+\infty))}.
\]
It is worth noting that, in the case of $t>T_{w}$, we exploited
 the easily verified inequality
\[
\int_{0}^{t}(t-\tau)^{-\theta}w_{1}(\tau)\frac{d w_{1}}{d
\tau}(\tau) d\tau \leq w_{1}(t)\bigg\|\frac{d w_{1}}{d
t}\bigg\|_{L_{\infty}([0,+\infty))}
\int_{0}^{T_{w}}(T_{w}-\tau)^{-\theta}d\tau.
\]
At last, collecting all estimates ends the proof of this statement.
\end{proof}

We complete this section with some results related to the function
$\mathcal{E}_{\bar{\beta},\beta_{0}}(t;\bar{d})$ introduced in
\eqref{4.1}.

\begin{lemma}\label{lA.1}
The function
$\mathcal{E}_{\bar{\beta},\beta_{0}}=\mathcal{E}_{\bar{\beta},\beta_{0}}(t;\bar{d})$
has the following properties.

\noindent\verb"(i)" $\mathcal{E}_{\bar{\beta},\beta_0}(t;\bar{d})$
is a completely monotonic function for $t>0$ if
$0<\beta_{i}<\beta_{0}<1,$ $i=1,2,...,m.$

\noindent\verb"(ii)" The following asymptotic expansions hold
\begin{align*}
\mathcal{E}_{\bar{\beta},\beta_0}(t;\bar{d})&\underset{t\to+\infty}{\sim}
\begin{cases}
\frac{t^{\beta_0-\beta_1-1}}{d_{1}\Gamma(\beta_0-\beta_1)},\qquad\quad\text{if}\qquad \beta_0\neq\beta_{1},\\
\\
-\frac{d_2t^{\beta_2-\beta_1-1}}{d_{1}\Gamma(\beta_2-\beta_1)},\qquad\text{if}\qquad \beta_0=\beta_{1},\\
\end{cases}\\
\mathcal{E}_{\bar{\beta},\beta_0}(t;\bar{d})&\underset{t\to
0}{\sim}\frac{t^{\beta_0-1}}{\Gamma(\beta_0)}-\sum_{i=1}^{m}d_{i}\frac{t^{\beta_0+\beta_{i}-1}}{\Gamma(\beta_{0}+\beta_{i})}.
\end{align*}

\noindent\verb"(iii)" For $\theta>0$ and the integer $n<\beta_0$,
there hold
\[
I_{t}^{\theta}\mathcal{E}_{\bar{\beta},\beta_0}(t)=\mathcal{E}_{\bar{\beta},\beta_0+\theta}(t;\bar{d}),\quad
\bigg(\frac{d}{d
t}\bigg)^{n}\mathcal{E}_{\bar{\beta},\beta_0}(t;\bar{d})=\mathcal{E}_{\bar{\beta},\beta_0-n}(t;\bar{d}).
\]

\noindent\verb"(iv)" Let $k=k(t)$ be a non-negative summable
function  satisfying the estimate
\[
k\leq C t^{-\theta},
\]
for some $\theta\in[0,1]$ and all $t\geq 0$.

We also assume that $\beta_{1}=\beta_0<1$.
 Then the function
$G(t)=\mathcal{E}_{\bar{\beta},\beta_0}*k$ is non-negative and
$G(t)\in L_1(\R^{+})$.

\noindent\verb"(v)" Let $\beta_1\leq \beta_0<1$, then the
following relations hold:
\begin{align*}
\mathcal{E}_{\bar{\beta},\beta_0}*1&=\mathcal{E}_{\bar{\beta},1+\beta_0}(t;\bar{d})\sim
\begin{cases}
\frac{t^{\beta_0-\beta_1}}{d_1\Gamma(1+\beta_0-\beta_1)},\qquad\qquad\qquad\text{if}\quad
t\to+\infty,\\
\\
\frac{t^{\beta_{0}}}{\Gamma(1+\beta_0)}-\sum_{i=1}^{m}\frac{d_{i}t^{\beta_0+\beta_i}}{\Gamma(1+\beta_0+\beta_{i})},\quad\text{if}\quad
t\to 0,
\end{cases}\\
\mathcal{E}_{\bar{\beta},\beta_0}*\omega_{1-\beta_{j}}&=\mathcal{E}_{\bar{\beta},1+\beta_0-\beta_{j}}(t;\bar{d})\sim
\begin{cases}
\frac{t^{\beta_0-\beta_1-\beta_{j}}}{d_1\Gamma(1+\beta_0-\beta_{j}-\beta_1)},
\qquad\qquad\qquad\text{if}\quad
t\to+\infty,\\
\\
\frac{t^{\beta_{0}-\beta_{j}}}{\Gamma(1+\beta_0-\beta_{j})}-\sum_{i=1}^{m}\frac{d_{i}t^{\beta_0-\beta_{j}+\beta_i}}
{\Gamma(1+\beta_0-\beta_{j}+\beta_{i})},\quad\text{if}\quad t\to 0,
\end{cases}\\
\end{align*}
for each $j=0,1,...,m.$
\end{lemma}
\begin{proof}
Statements \verb"(i)-(iii)" are obtained in Theorems 2.3-3.2 in
\cite{B} (see also representations (3.4) and (3.6) therein). Thus,
we are left to examine (iv) and (v).

Coming to statements in \verb"(iv)", we notice that the
non-negativity of $G(t)$ is a simple consequence of \verb"(i)" in
this claim and the non-negativity of $k$. Then employing the second
equality in \verb"(iii)", we rewrite $G(t)$ in more suitable form
\begin{align*}
G(t)&=\int_{0}^{t}\frac{d}{d
(t-\tau)}[(t-\tau)^{\beta_{0}}E_{(\beta_1,...,\beta_{m})1+\beta_{0}}(-d_{1}(t-\tau)^{\beta_1},...,-d_{m}(t-\tau)^{\beta_{m}})]k(\tau)d\tau\\
& =\frac{d}{d
t}\int_{0}^{t}(t-\tau)^{\beta_0}E_{(\beta_1,...,\beta_{m})1+\beta_{0}}(-d_{1}(t-\tau)^{\beta_1},...,-d_{m}(t-\tau)^{\beta_{m}})k(\tau)d\tau.
\end{align*}
Here we utilized the easily verified equality (see the second
asymptotic in \verb"(ii)" in this claim)
\[
\underset{t\to
0}{\lim}\,t^{\beta_{0}}E_{(\beta_1,...,\beta_{m})1+\beta_{0}}(-d_{1}t^{\beta_1},...,-d_{m}t^{\beta_{m}})=0.
\]
Then, bering in mind this representation and the non-negativity of
$G(t)$, we deduce
\[
\|G\|_{L_1(\R^{+})}=J_1+J_2,
\]
where we set
\begin{align*}
J_{1}&=-\underset{t\to
0}{\lim}\int_{0}^{t}\tau^{\beta_{0}}E_{(\beta_1,...,\beta_{m})1+\beta_{0}}(-d_{1}\tau^{\beta_1},...,-d_{m}\tau^{\beta_{m}})k(t-\tau)d\tau,\\
J_{2}&=\underset{t\to
+\infty}{\lim}\int_{0}^{t}\tau^{\beta_{0}}E_{(\beta_1,...,\beta_{m})1+\beta_{0}}(-d_{1}\tau^{\beta_1},...,-d_{m}\tau^{\beta_{m}})k(t-\tau)d\tau.
\end{align*}
Concerning the term $J_1$, we exploit the asymptotic expansion in
\verb"(ii)" and obtain the following estimates for any $0<t<<1$
\begin{align*}
|J_{1}|&\leq
C\int_{0}^{t}\bigg|\frac{\tau^{\beta_0}}{\Gamma(\beta_0+1)}-\sum_{i=1}^{m}\frac{d_{i}\tau^{\beta_{0}+\beta_{i}}}{\Gamma(\beta_{0}+1+\beta_{i})}\bigg|
k(t-\tau)d\tau\\
& \leq
C[t^{1+\beta_0}+\sum_{i=1}^{m}t^{1+\beta_0+\beta_{i}}]\|k\|_{L_{1}(\R^{+})}\underset{t\to
0}{\rightarrow}0
\end{align*}
which means vanishing $J_{1}$.

As for the term $J_2$, we rewrite it in the form
\begin{align*}
J_2&=\underset{t\to
+\infty}{\lim}\int_{\varepsilon}^{t}\tau^{\beta_{0}}E_{(\beta_1,...,\beta_{m})1+\beta_{0}}(-d_{1}\tau^{\beta_1},...,-d_{m}\tau^{\beta_{m}})k(t-\tau)d\tau\\
& + \underset{t\to
+\infty}{\lim}\int_{0}^{\varepsilon}\tau^{\beta_{0}}E_{(\beta_1,...,\beta_{m})1+\beta_{0}}(-d_{1}\tau^{\beta_1},...,-d_{m}\tau^{\beta_{m}})k(t-\tau)d\tau,
\end{align*}
where $\varepsilon$ is a some fixed small positive number,
$\varepsilon\in(0,1)$.

Finally, appealing to the asymptotic in \verb"(ii)", we end up
with
\[
J_{2}\leq
C[\varepsilon^{1+\beta_{0}}+\sum_{i=1}^{m}\varepsilon^{1+\beta_0+\beta_{i}}]\|k\|_{L_{1}(\R^{+})}+C\underset{t\to+\infty}{\lim}\int_{\varepsilon}^{t}k(t-\tau)d\tau
\leq C\|k\|_{L_{1}(\R^{+})}.
\]
In fine, collecting estimates of $J_1$ and $J_2$ ends the
validation of statement \verb"(iv)".

At this point, we verify estimates in \verb"(v)". Concerning the
first asymptotic, we easily conclude that
\[
(\mathcal{E}_{\bar{\beta},\beta_{0}}*1)(t)=I_{t}^{1}\mathcal{E}_{\bar{\beta},\beta_{0}}(t;\bar{d})
= \mathcal{E}_{\bar{\beta},\beta_{0}+1}(t;\bar{d}).
\]
Then,  the asymptotic \verb"(ii)" provides the following behavior
\[
\mathcal{E}_{\bar{\beta},\beta_0+1}(t;\bar{d})\sim
\begin{cases}
\frac{t^{\beta_{0}-\beta_{1}}}{d_1\Gamma(1+\beta_0-\beta_{1})},\qquad\qquad\text{if}\quad
t\to+\infty,\\
\\
\frac{t^{\beta_{0}}}{\Gamma(1+\beta_{0})}-\sum_{i=1}^{m}\frac{d_{i}t^{\beta_{0}+\beta_{i}}}{\Gamma(1+\beta_{0}+\beta_{i})},\quad\text{if}\quad
t\to 0,
\end{cases}
\]
which completes the verification of the first relation in
\verb"(v)".

As for the remaining equality, it is  a simple consequence of the
following relations (which follow from \verb"(ii)")
\[
\mathcal{E}_{\bar{\beta},\beta_{0}}*\omega_{1-\beta_{j}}=I_{t}^{1-\beta_{j}}\mathcal{E}_{\bar{\beta},\beta_{0}}(t;\bar{d})=\mathcal{E}_{\bar{\beta},1+\beta_{0}-\beta_{j}}(t;\bar{d})
\]
 and of the asymptotic in
\verb"(iii)".
 This finishes the proof of this claim.
\end{proof}



\section{Proof of Theorem \ref{t3.1}}
\label{s5}

\noindent Here we provide the detailed proof of this theorem in the
\textbf{DBC} case. The case of \textbf{NBC} is analyzed in the
similar manner and is left to the interested readers. Besides, in
our analysis we focus on the most difficult case, i.e. $\rho$ and
$\rho_{i}$ are not constants. It is apparent that, the estimate
\eqref{3.2*} is a simple consequence of \eqref{3.2}. Indeed,
collecting \eqref{3.2} with the  Sobolev embedding theorem provides
the estimate of the first term in the left-hand side of
\eqref{3.2*}. The second term in the left-hand side in \eqref{3.2*}
is managed with \eqref{3.2} and the Young's inequality for the
convolution.  Thus, we are left to evaluate
$\|u\|_{L_{\infty}([0,+\infty),W^{1,2}(\Omega))}$.

Here, we will exploit the strategy containing  three main steps. In
the first stage, we aim to evaluate  \eqref{3.2} for
$t\in[0,2T^{*}]$ with
\[
T^{*}=\max\{1,T_{0},T_{1},...,T_{M}\}.
\]
To this end, we multiply the equation in \eqref{i.1} by the
function
\[
v=u(x,t)-\frac{\partial^{2} u}{\partial x^{2}}(x,t),
\]
and then integrate over $\Omega$. Performing technical
calculations, we  reduce \eqref{i.1}, \eqref{i.2} to the following
differential inequality in an unknown positive function
$V(t)=\int_{\Omega}\big[u^{2}+\big(\frac{\partial u}{\partial
x}\big)^{2}\big]dx$:
\begin{equation}\label{5.1*}
\frac{\rho(t)}{2}\mathbf{D}_{t}^{\nu}V+\sum_{i=1}^{M}\frac{\rho_{i}(t)}{2}\mathbf{D}_{t}^{\nu_{i}}V-\mathcal{K}_{0}*V+C
V\leq F(t)
\end{equation}
with the positive constant $C$ and the positive  function $F$
depending on $V(0),$ $\omega_{1-\nu_{i}}(t),$ $i=1,...,M;$  the
parameters of the model, and the corresponding norms of the
coefficients. The positive kernel $\mathcal{K}_{0}$ depends only on
$\mathcal{K}$ and the norms of the coefficients of the operator
$\mathcal{L}_{2}$. After that, exploiting \eqref{5.1*} and the
Gronwall-type inequality, we completes this step with bound
\eqref{3.2} for $t\in[0,2T^{*}]$.

In the second step, introducing the special cut-off function
$\xi(t)$ vanishing for $t\geq 2T^{*}$ and performing the change of
time-variable, we transform inequality \eqref{5.1*} to the Cauchy
problem for the  fractional ordinary differential equation
(\textbf{FODE}) with constant coefficients $d_{i}$ with unknown
function $\mathcal{V}=(1-\xi)V$:
\begin{equation}\label{5.1}
\begin{cases}
\mathbf{D}_{t}^{\nu}\mathcal{V}+\sum_{i=1}^{M}d_{i}\mathbf{D}_{t}^{\nu_{i}}\mathcal{V}-\mathcal{K}_{0}*\mathcal{V}+d_{0}\mathcal{V}=\mathcal{F}(t),\,
t>0,\\
\mathcal{V}(0)=0.
\end{cases}
\end{equation}
Finally, using Lemma \ref{lA.1} in this art and
\cite[Theorem 4.1]{LG}, we obtain the estimate
$\|\mathcal{V}\|_{L_{\infty}([0,+\infty))}$, which in turn ensures
\eqref{3.2}  and finishes the proof of this claim.

We notice that, in our arguments (after integrating by parts  in the
equation in \eqref{i.1}) we use the extra regularity of the function
$u,$ namely, $\mathbf{D}_{t}^{\nu}\int_{\Omega}\big(\frac{\partial
u}{\partial x}\big)^{2}dx,$
$\mathbf{D}_{t}^{\nu_{i}}\int_{\Omega}\big(\frac{\partial
u}{\partial x}\big)^{2}dx.$ However,  we do not need in this
smoothness in the final estimates. Thus, our computation is,
partially,  formal but  can be rigorously justified via  the
well-known technique employed in \cite[p.248]{LSU} in the case of
integer derivatives, and in \cite[Lemma 4.2]{KPV2}, \cite[Lemma
5.7]{SV} in the case of the fractional derivatives. We recall that,
this approach is concerned with the exploiting sequence of functions
 having extra regularity in Step 1 (in
our case, using the mollification of $u$) and then, taking
advantage of the standard approximation arguments (see e.g.,
\cite[Section 1]{AF})á passing to the limit in the final estimate.
However, for simplicity considerations, we omit this in the course
of our investigation.

\noindent\textit{Step 1.} Multiplying the equation in \eqref{i.1}
by $v$ and integrating over $\Omega$ arrive at the identity
\begin{equation}\label{5.3}
\sum_{i=1}^{4}\mathcal{U}_{i}=0,
\end{equation}
where we set
\[
\mathcal{U}_{1}=-\int_{\Omega}v\mathcal{L}_{1}udx,\quad
\mathcal{U}_{2}=-\int_{\Omega}v\mathcal{K}*\mathcal{L}_{2}u
dx,\quad \mathcal{U}_{3}=\int_{\Omega}f(u)vdx,\quad
\mathcal{U}_{4}=\int_{\Omega}v\mathbf{D}_{t}u dx.
\]
At this point, we evaluate each term $\mathcal{U}_{j}$, separately.

\noindent$\bullet$ Integrating by parts  and taking into account
the homogenous \textbf{DBC} and  the smoothness of the
coefficients (see h3), we end up with the relation
\[
\mathcal{U}_{1}=\int_{\Omega}[a_1-a_0]\bigg(\frac{\partial
u}{\partial
x}\bigg)^{2}dx-\int_{\Omega}a_0u^{2}dx+\int_{\Omega}a_{1}\bigg(\frac{\partial^{2}
u}{\partial x^{2}}\bigg)^{2}dx+\int_{\Omega}\frac{\partial
a_{1}}{\partial x}\frac{\partial u}{\partial
x}\frac{\partial^{2}u}{\partial
x^{2}}dx-\int_{\Omega}\frac{\partial a_{0}}{\partial
x}u\frac{\partial u}{\partial x}dx.
\]
Clearly, if $\frac{\partial a_{1}}{\partial x}\equiv 0$ and
$\frac{\partial a_0}{\partial x}\equiv 0$, then the estimate of
$\mathcal{U}_{1}$ is finished. To complete this evaluation in the
general case (i.e.  $\frac{\partial a_{1}}{\partial x}\not\equiv
0$ and $\frac{\partial a_0}{\partial x}\not\equiv 0$), we exploit
the Cauchy inequality to manage the  last two terms
\begin{align*}
\mathcal{U}_{1}&\geq
\int_{\Omega}\bigg[a_1-a_0-\varepsilon_1-\varepsilon_{0}^{-1}\bigg\|\frac{\partial
a_{1}}{\partial
x}\bigg\|^{2}_{L_{\infty}([0,+\infty),\C(\bar{\Omega}))}\bigg]\bigg(\frac{\partial
u}{\partial x}\bigg)^{2}dx\\
&+
\int_{\Omega}\bigg[-a_{0}-\varepsilon_{1}^{-1}\bigg\|\frac{\partial
a_{0}}{\partial
x}\bigg\|^{2}_{L_{\infty}([0,+\infty),\C(\bar{\Omega}))}\bigg]u^{2}dx
+
\int_{\Omega}[a_1-\varepsilon_{0}]\bigg(\frac{\partial^{2}u}{\partial
x^{2}}\bigg)^{2}dx.
\end{align*}
Then, bearing in mind assumptions h2 and \eqref{3.1}, and setting
\[
\varepsilon_{0}=\delta^{*},\quad
\varepsilon_{1}=\frac{\delta^{*}(\delta_{1}+\delta_{0})
-\|\frac{\partial a_{1}}{\partial x}
\|^{2}_{L_{\infty}([0,+\infty),\C(\bar{\Omega}))}}{4\delta^{*}},
\]
we end up with the bound
\begin{align*}
\mathcal{U}_{1}&\geq
3\varepsilon_{1}\int_{\Omega}\bigg(\frac{\partial u}{\partial
x}\bigg)^{2}dx+[\delta_{1}-\delta^{*}]\int_{\Omega}\bigg(\frac{\partial^{2}
u}{\partial x^{2}}\bigg)^{2}dx\\
&+ \bigg[ \delta_{0}-\varepsilon_{1}^{-1}\bigg\|\frac{\partial
a_{0}}{\partial
x}\bigg\|^{2}_{L_{\infty}([0,+\infty),\C(\bar{\Omega}))}
\bigg]\int_{\Omega}u^{2}dx.
\end{align*}

\noindent$\bullet$ Taking advantage of the Cauchy inequality, and
appealing to the regularity of $b_{1},b_{2}$ and $\mathcal{K}$
(see assumptions h3, h4), we have
\begin{align*}
\mathcal{U}_{2}&\geq -[\varepsilon_{3}^{-1}+\varepsilon_{4}^{-1}]
\sum_{l=0}^{1}\|b_{l}\|^{2}_{L_{\infty}([0,+\infty),\C(\bar{\Omega}))}
\int_{\Omega}|\mathcal{K}|*\bigg[u^{2}+\bigg(\frac{\partial
u}{\partial x}\bigg)^{2}\bigg]dx\\
& - \varepsilon_{3}
\|\mathcal{K}\|_{L_{1}(\R^{+})}\int_{\Omega}u^{2}dx-
\varepsilon_{4}
\|\mathcal{K}\|_{L_{1}(\R^{+})}\int_{\Omega}\bigg(\frac{\partial^{2}
u}{\partial x^{2}}\bigg)^{2}dx.
\end{align*}
Then, setting
\[
\varepsilon_{3}=\frac{\delta_{0}}{8
\|\mathcal{K}\|_{L_{1}(\R^{+})}}\qquad\text{and}\qquad
\varepsilon_{4}=\frac{\delta_{1}-\delta^{*}}{8
\|\mathcal{K}\|_{L_{1}(\R^{+})}}, \]
 we deduce
\begin{align*}
\mathcal{U}_{2}&\geq
-8\sum_{l=0}^{1}\|b_{l}\|^{2}_{L_{\infty}([0,+\infty),\C(\bar{\Omega}))}
\|\mathcal{K}\|_{L_{1}(\R^{+})}\bigg[\frac{1}{\delta_{0}}+\frac{1}{\delta_{1}-\delta^{*}}\bigg]\int_{\Omega}|\mathcal{K}|*\bigg[u^{2}+\bigg(\frac{\partial
u}{\partial x}\bigg)^{2}\bigg]dx\\
&-
\frac{\delta_{0}}{8}\int_{\Omega}u^{2}dx-\frac{\delta_1-\delta^{*}}{8}\int_{\Omega}\bigg(\frac{\partial^{2}u
}{\partial x^{2}}\bigg)^{2}dx.
\end{align*}

\noindent$\bullet$ Concerning the term $\mathcal{U}_{3}$, we first
rewrite it in more suitable form
\[
\mathcal{U}_{3}=\int_{\Omega}f(u)udx+
\int_{\Omega}[L_{4}u-f(0)]\frac{\partial^{2}u }{\partial x^{2}}dx
-\int_{\Omega}[f(u)-f(0)+L_{4}u]\frac{\partial^{2}u }{\partial
x^{2}}dx.
\]
We notice that this step is omitted in the case of the homogeneous
Neumann boundary condition.

After that, bearing in mind the easily verified equality
\[
[f(u)-f(0)+L_{4}u]=0\quad \text{on}\quad\partial\Omega_{T},
\]
we integrate by parts in the last  integral and then apply
condition h5 to handle the first two integrals, we obtain
\begin{align*}
\mathcal{U}_{3}&\geq
-L_{2}|\Omega|+L_{3}\int_{\Omega}|u|^{1+\gamma}dx+\int_{\Omega}\big(\frac{\partial
u}{\partial
x}\big)^{2}[f'(u)+L_{4}]dx+\int_{\Omega}L_{4}u\frac{\partial^{2}u}{\partial
x^{2}}dx -L_{1}\int_{\Omega}\frac{\partial^{2}u}{\partial x^{2}}dx
\\
&
\geq-L_{2}|\Omega|+L_{3}\int_{\Omega}|u|^{1+\gamma}dx+L_{4}\int_{\Omega}u\frac{\partial^{2}u}{\partial
x^{2}}dx -L_{1}\int_{\Omega}\frac{\partial^{2}u}{\partial
x^{2}}dx.
\end{align*}
Taking into account that $\gamma\geq 1$ (see h5) and utilizing the
Young inequality, we  conclude that
\[
\int_{\Omega}|u|^{2}dx\leq
\frac{2\varepsilon}{1+\gamma}\int_{\Omega}|u|^{1+\gamma}dx+\frac{\gamma-1}{1+\gamma}\varepsilon^{-\frac{2(1+\gamma)}{\gamma-1}}|\Omega|
\]
with a positive $\varepsilon$ being specified below.

Exploiting this inequality, we end up with the estimate
\[
\mathcal{U}_{3}\geq
-L_{2}|\Omega|+L_{3}\frac{1+\gamma}{2\varepsilon}\int_{\Omega}u^{2}dx
-\frac{\gamma-1}{2}\varepsilon^{-\frac{3\gamma+1}{\gamma-1}}L_{3}|\Omega|
+L_{4}\int_{\Omega}u\frac{\partial^{2}u}{\partial x^{2}}dx
-L_{1}\int_{\Omega}\frac{\partial^{2}u}{\partial x^{2}}dx.
\]
Then, applying the Cauchy inequality to manage the last two terms
in the right-hand side arrives at the estimate
\[
\mathcal{U}_{3}\geq
-|\Omega|\bigg[L_{2}+\frac{\gamma-1}{2}\varepsilon^{-\frac{3\gamma+1}{\gamma-1}}L_{3}+\frac{8L_{1}^{2}}{\delta_{1}-\delta^{*}}
\bigg] + \bigg[ \frac{1+\gamma}{2\varepsilon}L_{3}
-\frac{8L_{4}^{2}}{\delta_{1}-\delta^{*}}
\bigg]\int_{\Omega}u^{2}dx -
\frac{\delta_{1}-\delta^{*}}{4}\int_{\Omega}\bigg(\frac{\partial^{2}
u}{\partial x^{2}}\bigg)^{2}dx.
\]
Setting
\[
\varepsilon^{-1}=\frac{2}{L_{3}(1+\gamma)}\bigg[\frac{8L_{4}^{2}}{\delta_{1}-\delta^{*}}+
\frac{1}{\varepsilon_{1}}\bigg\|\frac{\partial a_{0}}{\partial
x}\bigg\|^{2}_{L_{\infty}([0,+\infty),\C(\bar{\Omega}))}
 \bigg]
\]
and
\[
C_{2}=|\Omega|\bigg[L_{2}+L_{3}\varepsilon^{-\frac{3\gamma+1}{\gamma-1}}\frac{\gamma-1}{2}+\frac{8L_{1}^{8}}{\delta_{1}-\delta^{*}}\bigg],
\]
we arrive at
\[
\mathcal{U}_{3}\geq -C_{2}+\varepsilon_{1}^{-1}\bigg\|\frac{\partial
a_{0}}{\partial
x}\bigg\|^{2}_{L_{\infty}([0,+\infty),\C(\bar{\Omega}))}\int_{\Omega}u^{2}dx
-
\frac{\delta_{1}-\delta^{*}}{4}\int_{\Omega}\bigg(\frac{\partial^{2}
u}{\partial x^{2}}\bigg)^{2}dx.
\]

\noindent$\bullet$ It is apparent  that the bound of
 $\mathcal{U}_{4}$ follows from the
 proper estimates of the functions:
\[i_{1}=\int_{\Omega}u\mathbf{D}_{t}^{\theta}(\rho_{\theta}u)dx,\qquad
i_{2}=-\int_{\Omega}\frac{\partial^{2} u}{\partial x^{2}}
\mathbf{D}_{t}^{\theta}(\rho_{\theta}u)dx,
\]
where
\[
\rho_{\theta}=
\begin{cases}
\rho,\qquad\text{if}\quad \theta=\nu, \\
\rho_{i},\qquad\text{if}\quad \theta=\nu_{i},\quad i=1,2,...,M.
\end{cases}
\]
Appealing to the definition of the Caputo derivative and taking
into account \eqref{2.2}, we rewrite $i_{1}$ in more convenient
form
\[
i_{1}=\rho^{-1}_{\theta}(t)\int_{\Omega}\rho_{\theta}(t)u\partial_{t}^{\theta}(\rho_{\theta}u)dx-\omega_{1-\theta}(t)\int_{\Omega}\rho_{\theta}(0)u_{0}udx.
\]
Exploiting Proposition \ref{pA.1} to control the first term in the
right-hand side, we arrive at the inequalities:
\begin{align*}
i_{1}&\geq
\frac{1}{2\rho_{\theta}(t)}\int_{\Omega}\partial_{t}^{\theta}(\rho_{\theta}^{2}u^{2})dx+\frac{\rho_{\theta}(t)\omega_{1-\theta}(t)}{2}\int_{\Omega}u^{2}dx-
\omega_{1-\theta}(t)\rho_{\theta}(0)\int_{\Omega}u_{0}udx\\
& \geq
\frac{1}{2\rho_{\theta}(t)}\int_{\Omega}\partial_{t}^{\theta}(\rho_{\theta}^{2}u^{2})dx+\frac{\rho_{\theta}(0)\omega_{1-\theta}(t)}{2}\int_{\Omega}u^{2}dx
- \omega_{1-\theta}(t)\rho_{\theta}(0)\int_{\Omega}u_{0}udx.
\end{align*}
Here, to obtain the last inequality, we used assumptions h3 and h2
on $\rho_{\theta}$.

Finally, taking advantage of the Cauchy inequality to manage the
last term in the  right-hand side arrives at the bound
\[
i_{1}\geq\frac{1}{2\rho_{\theta}(t)}\int_{\Omega}\partial_{t}^{\theta}(\rho_{\theta}^{2}u^{2})dx
-
\frac{\rho_{\theta}(0)\omega_{1-\theta}(t)}{2}\int_{\Omega}u_{0}^{2}dx.
\]
To calculate the fractional derivative in the first integral, we
employ Proposition \ref{pA.2} and deduce
\[
i_{1}\geq
\frac{\rho_{\theta}(t)}{2}\int_{\Omega}\mathbf{D}_{t}^{\theta}u^{2}dx+
\bigg[\frac{\partial_{t}^{\theta}\rho_{\theta}^{2}}{2\rho_{\theta}(t)}-\frac{\rho_{\theta}(0)\omega_{1-\theta}(t)}{2}\bigg]\int_{\Omega}u_{0}^{2}dx
+
\frac{\theta}{2\rho_{\theta}(t)\Gamma(1-\theta)}\int_{\Omega}\mathcal{J}_{\theta}(t;\rho_{\theta}^{2},u^{2})dx,
\]
where
\[
\mathcal{J}_{\theta}(t;\rho_{\theta}^{2},u^{2})=\int_{0}^{t}\frac{[\rho_{\theta}^{2}(t)-\rho_{\theta}^{2}(s)][u^{2}(s)-u^{2}(0)]}{(t-s)^{1+\theta}}ds.
\]
At last, employing Corollaries \ref{cA.3} and \ref{cA.4} to
evaluate the second and the third term, we obtain
\begin{equation}\label{5.3*}
i_{1}\geq
\frac{\rho_{\theta}(t)}{2}\int_{\Omega}\mathbf{D}_{t}^{\theta}u^{2}dx-\bigg[
\rho_{\theta}(0)\omega_{1-\theta}(t)+T_{\rho_{\theta}}^{1-\theta}\frac{\|\frac{d\rho_{\theta}}{d
t}\|_{L_{\infty}([0,+\infty))}}{\Gamma(2-\theta)}(1+C\theta)
\bigg]\int_{\Omega}u_{0}^{2}dx.
\end{equation}
Working under the homogeneous \textbf{DBC} and integrating
 by parts in $i_{2}$ arrive at the relations
\begin{align*}
i_{2}&=\int_{\Omega}\frac{\partial u}{\partial
x}\partial_{t}^{\theta}\bigg(\rho_{\theta}\frac{\partial
u}{\partial x}-\rho_{\theta}(0)\frac{\partial u_{0}}{\partial
x}\bigg)dx  =
\frac{1}{\rho_{\theta}(t)}\int_{\Omega}\frac{\partial}{\partial
x}\big(\rho_{\theta}u\big)\partial_{t}^{\theta}\frac{\partial}{\partial
x}\big(\rho_{\theta}u\big)dx\\
& -\omega_{1-\theta}(t)\int_{\Omega}\rho_{\theta}(0)\frac{\partial
u_{0}}{\partial x}\frac{\partial u}{\partial x}dx.
\end{align*}
Then, recasting the arguments leading to estimate \eqref{5.3*}, we
deduce that
\begin{equation}\label{5.4}
i_{2}\geq
\frac{\rho_{\theta}(t)}{2}\int_{\Omega}\mathbf{D}_{t}^{\theta}\bigg(\frac{\partial
u}{\partial x}\bigg)^{2}dx -\bigg[
\rho_{\theta}(0)\omega_{1-\theta}(t)+T_{\rho_{\theta}}^{1-\theta}(C\theta+1)\frac{\|\rho'_{\theta}\|_{L_{\infty}([0,+\infty))}}{\Gamma(2-\theta)}
\bigg]\int_{\Omega}\bigg(\frac{\partial u_{0}}{\partial
x}\bigg)^{2}dx.
\end{equation}
Finally, using estimates \eqref{5.3*} and \eqref{5.4} and performing
technical calculations, we arrive at
\begin{align*}
\mathcal{U}_{4}&\geq
\frac{1}{2}\bigg\{\rho(t)\mathbf{D}_{t}^{\nu}\int_{\Omega}\bigg[u^{2}+\bigg(\frac{\partial
u}{\partial
x}\bigg)^{2}\bigg]dx+\sum_{i=1}^{M}\rho_{i}(t)\mathbf{D}_{t}^{\nu_{i}}\int_{\Omega}\bigg[u^{2}+\bigg(\frac{\partial
u}{\partial x}\bigg)^{2}\bigg]dx\bigg\}\\
&
-\bigg\{\rho(0)\omega_{1-\nu}(t)+\sum_{i=1}^{M}\rho_{i}(0)\omega_{1-\nu_{i}}(t)+T_{0}^{1-\nu}(C\nu+1)\frac{\|\rho'\|_{L_{\infty}([0,+\infty)}}{\Gamma(2-\nu)}
\\&
+ \sum_{i=1}^{M}
T_{i}^{1-\nu_{i}}(C\nu_{i}+1)\frac{\|\rho_{i}'\|_{L_{\infty}([0,+\infty))}}{\Gamma(2-\nu_{i})}
 \bigg\}\|u_{0}\|_{W^{1,2}(\Omega)}^{2}.
\end{align*}
In fine, setting
\[
C_{3}=(T^{*})^{1-\nu_{1}}\frac{C\nu+1}{\Gamma(2-\nu)}\bigg[
\|\rho'\|_{L_{\infty}([0,+\infty))}
+\sum_{i=1}^{M}\|\rho_{i}'\|_{L_{\infty}([0,+\infty))}\bigg],
\]
we end up with the bound
\begin{align*}
\mathcal{U}_{4}&\geq \frac{1}{2}\bigg\{
\rho(t)\mathbf{D}_{t}^{\nu}\int_{\Omega}\bigg[u^{2}+\bigg(\frac{\partial
u}{\partial x}\bigg)^{2}\bigg]dx +
\sum_{i=1}^{M}\rho_{i}(t)\mathbf{D}_{t}^{\nu_{i}}\int_{\Omega}\bigg[u^{2}+\bigg(\frac{\partial
u}{\partial x}\bigg)^{2}\bigg]dx  \bigg\}\\
& -\bigg[
C_{3}+\rho(0)\omega_{1-\nu}(t)+\sum_{i=1}^{M}\rho_{i}(0)\omega_{1-\nu_{i}}(t)
\bigg]\|u_{0}\|^{2}_{W^{1,2}(\Omega)}.
\end{align*}
Now, bearing in mind \eqref{5.3} and collecting all estimates of $\mathcal{U}_{i}$, we get
\begin{align*}
&3\varepsilon_{1}\int_{\Omega}\bigg(\frac{\partial u}{\partial
x}\bigg)^{2}dx+\frac{5}{8}[\delta_{1}-\delta^{*}]\int_{\Omega}\bigg(\frac{\partial^{2}u}{\partial
x^{2}}\bigg)^{2}dx+\frac{7\delta_{0}}{8}\int_{\Omega}u^{2}dx\\
&
-8\sum_{i=0}^{1}\|b_i\|^{2}_{L_{\infty}([0,+\infty),\C(\bar{\Omega}))}\|\mathcal{K}\|_{L_{1}(\R^{+})}[(\delta_1-\delta^{*})^{-1}+\delta_{0}^{-1}]
\int_{\Omega}|\mathcal{K}|*\bigg[u^{2}+\bigg(\frac{\partial
u}{\partial x}\bigg)^{2}\bigg]dx\\
& +
\frac{\rho(t)}{2}\mathbf{D}_{t}^{\nu}\int_{\Omega}\bigg[u^{2}+\bigg(\frac{\partial
u}{\partial x}\bigg)^{2}\bigg]dx
+\sum_{i=1}^{M}\frac{\rho_{i}(t)}{2}\mathbf{D}_{t}^{\nu_{i}}\int_{\Omega}\bigg[u^{2}+\bigg(\frac{\partial
u}{\partial x}\bigg)^{2}\bigg]dx\\
& \leq
C_{2}+\|u_{0}\|^{2}_{W^{1,2}(\Omega)}\bigg[\rho(0)\omega_{1-\nu}(t)+\sum_{i=1}^{M}\rho_{i}(0)\omega_{1-\nu_{i}}(t)+C_{3}\bigg].
\end{align*}
At last, denoting
\begin{align}\label{5.7*}\notag
C_{4}&=\min\{3\varepsilon_{1},7\delta_{0}/8\},\\\notag
\mathcal{K}_{0}(t)&=8\sum_{i=0}^{1}\|b_i\|^{2}_{L_{\infty}([0,+\infty),\C(\bar{\Omega}))}\|\mathcal{K}\|_{L_{1}(\R^{+})}[(\delta_1-\delta^{*})^{-1}+\delta_{0}^{-1}]
|\mathcal{K}(t)|,\\
F(t)&=C_{2}+\|u_{0}\|^{2}_{W^{1,2}(\Omega)}\bigg[\rho(0)\omega_{1-\nu}(t)+\sum_{i=1}^{M}\rho_{i}(0)\omega_{1-\nu_{i}}(t)+C_{3}\bigg].
\end{align}
we reach to the estimate
\begin{equation}\label{5.7}
\frac{\rho(t)}{2}\mathbf{D}_{t}^{\nu}V
+\sum_{i=1}^{M}\frac{\rho_{i}(t)}{2}\mathbf{D}_{t}^{\nu_{i}}V-\mathcal{K}_{0}*V+C_{4}V\leq
F(t)
\end{equation}
with the non-negative function
\[
V(t)=\int_{\Omega}\bigg[u^{2}+\bigg(\frac{\partial u}{\partial
x}\bigg)^{2}\bigg]dx\qquad\text{such that}\quad
V(0)=\int_{\Omega}\bigg[u_{0}^{2}+\bigg(\frac{\partial
u_{0}}{\partial x}\bigg)^{2}\bigg]dx.
\]
Finally, to obtain estimate \eqref{3.2} for $t\in[0,2T^{*}]$, we
compute the fractional integral $I_{t}^{\nu}$ in the both sides of
\eqref{5.7}. Combining standard technical calculations with
Proposition \ref{pA.2}, we arrive at the estimate
\[
\rho(t)
V(t)+\sum_{i=1}^{M}I^{\nu-\nu_{i}}_{t}(\rho_{i}V)(t)+C_{4}I^{\nu}_{t}V\leq
\sum_{i=1}^{3}J_{i},
\]
where we set
\begin{align*}
J_{1}&=2I_{t}^{\nu}F(t)+V(0)[\rho(t)+\sum_{i=1}^{M}I_{t}^{\nu-\nu_{i}}\rho_{i}(t)],\quad
J_{2}=2I_{t}^{\nu}(\mathcal{K}_{0}*V)(t),\\
J_{3}&=\nu
I^{1}_{t}(\mathcal{W}(\rho)V)(t)+\sum_{i=1}^{M}\nu_{i}I^{1+\nu-\nu_{i}}(\mathcal{W}(\rho_{i})V)(t).
\end{align*}
Then, taking into account the  non-negativity of
$I_{t}^{\nu-\nu_{i}}(\rho_{i}V)$, $I_{t}^{\nu}V$ and assumption h2
on $\rho(t)$, we end up with the bound
\begin{equation}\label{5.9}
V\leq \frac{1}{\delta}\sum_{i=1}^{3}J_{i}.
\end{equation}
At this point, we evaluate each term $J_{i},$ separately.

\noindent$\bullet$ Performing the straightforward calculations and
bearing in mind of Proposition \ref{pA.0}, assumption h3 and
definition of the function $F$, we deduce
\[
J_{1}\leq
C[1+V(0)][1+T^{*\nu}+T^{*}+\sum_{i=1}^{M}(T^{*\nu-\nu_{i}}+T^{*\nu-\nu_{i}+1})]
\]
with the positive constant $C$ depending only on $C_{2},C_{3},$
$\nu$ and $\nu_{i}$, and the norms of $\rho,\rho_{i}$.

\noindent$\bullet$ Employing (ii) in Proposition \ref{pA.0} and
properties of the kernel $\mathcal{K}_{0}$, we easily conclude
that
\[
0<J_{2}\leq 2(I_{t}^{\nu}\mathcal{K}_{0}*V)(t),
\]
and
\[
\|I_{t}^{\nu}\mathcal{K}_{0}\|_{L_{1}(0,2T^{*})}\leq
\frac{CT^{*\nu}}{\Gamma(1+\nu)}\|\mathcal{K}\|_{L_{1}(\R^{+})},
\]
where $C$ depending only on the norms of $b_{0},b_{1},\mathcal{K}$
and being independent of $T^{*}$.

\noindent$\bullet$ Performing technical calculations arrives at
\[
J_{3}\leq C
I_{t}^{1-\nu}(V[\mathcal{W}(\rho)T^{*\nu}+\sum_{i=1}^{M}\mathcal{W}(\rho_{i})T^{*2\nu-\nu_{i}}])(t)
\]
with $C$ depending only on $\nu$ and $\nu_{i}$ and being
independent of $t$ and $T^{*}$.

Collecting all estimates of $J_{i}$ with \eqref{5.9} provides the
inequality
\begin{align*}
V(t)&\leq
\frac{C}{\delta}[1+V(0)][1+T^{*\nu}+T^{*}+\sum_{i=1}^{M}(T^{*\nu-\nu_{i}}+T^{*\nu-\nu_{i}+1})]\\
&+\frac{2}{\delta}(I_{t}^{\nu}\mathcal{K}_{0})*V+
\frac{C}{\delta}I_{t}^{1-\nu}(V[\mathcal{W}(\rho)T^{*\nu}+\sum_{i=1}^{M}\mathcal{W}(\rho_{i})T^{*2\nu-\nu_{i}}])(t).
\end{align*}
In fine,  bearing in mind Proposition \ref{pA.2} and employing
Gronwall-type inequality in Proposition \ref{pA.0} with
\begin{align*}
k_{1}&=\mathcal{W}(\rho)T^{*\nu}+\sum_{i=1}^{M}\mathcal{W}(\rho_{i})T^{*2\nu-\nu_{i}},\quad
\theta=1-\nu,\quad k_{2}=I_{t}^{\nu}\mathcal{K}_{0},\\
\mathfrak{C}_{1}&=\mathfrak{C}_{2}=\frac{C}{\delta},\quad
\mathfrak{C}_{0}=\frac{C}{\delta}[1+V(0)][1+T^{*\nu}+T^{*}+\sum_{i=1}^{M}(T^{*\nu-\nu_{i}}+T^{*\nu-\nu_{i}+1})],
\end{align*}
we end up with the desired estimate \eqref{3.2} for all
$t\in[0,2T^{*}]$. Hence, the Step 1 is completed.

\noindent\textit{Step 2}. As we wrote above, to claim estimate
\eqref{3.2} for all $t\geq 0$, we first reduce inequality
\eqref{5.7} to the initial-value problem \eqref{5.1}.

To this end,  introducing  the new  functions $\xi$, $\mathcal{U}$
and  the operator $\mathfrak{S}$:
 \begin{align*}
\xi&=\xi(t)\in\C_{0}^{\infty}(\overline{R^{+}}),\quad
\xi\in[0,1],\quad \xi=\begin{cases}1\qquad 0\leq t\leq T^{*},\\
0,\quad\quad t\geq 2T^{*}, \end{cases}\\
 \mathcal{U}&=\xi(t) V(t),
\\
\mathfrak{S} &=
\frac{\rho(t)}{2}\mathbf{D}_{t}^{\nu}+\sum_{i=1}^{M}\frac{\rho_{i}(t)}{2}\mathbf{D}_{t}^{\nu_{i}}-\mathcal{K}_{0}*
+C_{4},
\end{align*}
we establish the following result.
\begin{proposition}\label{p5.1}
The function $\mathcal{U}$ boils down to $V$ for $t\in[0,T^{*}]$
and satisfies estimates \eqref{3.2} for all $t\geq 0$, besides
\begin{equation}\label{5.10}
|\mathfrak{S}\mathcal{U}|\leq \xi(t)F(t)+C[1+V(0)]\leq
F(t)+C[1+V(0)]
\end{equation}
with the positive constant $C$ depending only on
$T^{*},\nu,\nu_{i}$ and the corresponding norms of
$\rho,\rho_{i},\mathcal{K}_{0}$ and being independent of $t$.
\end{proposition}
\begin{proof}
The first statement of this claim follows immediately from the
definition of $\mathcal{U}$ and properties of $\xi(t)$. Thus, we
are left to examine \eqref{5.10}. To this end, we rewrite the
left-hand side of \eqref{5.10} in the form
\begin{equation}\label{5.11}
\mathfrak{S}\mathcal{U}=\xi(t)\mathfrak{S}V+\sum_{j=1}^{3}V_{j},
\end{equation}
where we set
\begin{align*}
V_{1}&=V(0)[\rho(t)\mathbf{D}_{t}^{\nu}\xi+\sum_{i=1}^{M}\rho_{i}\mathbf{D}_{t}^{\nu_{i}}\xi],\\
V_{2}&=\frac{\rho(t)\nu}{\Gamma(1-\nu)}\mathcal{J}_{\nu}(t;\xi,V)+\sum_{i=1}^{M}
\frac{\nu_{i}\rho_{i}(t)}{\Gamma(1-\nu_{i})}\mathcal{J}_{\nu_{i}}(t;\xi,V),\\
V_{3}&=\xi(\mathcal{K}_{0}*V)-(\mathcal{K}_{0}*\xi V).
\end{align*}
We notice that the terms $V_{1}$ and $V_{2}$ appears in
\eqref{5.11} thanks to Proposition \ref{pA.2}.

Next, we estimate each $V_{i}$ and $\xi\mathfrak{S}V$, separately.
Obviously, keeping in mind \eqref{5.7} and the definition of
$\xi$, we arrive at the inequality
\[
\xi\mathfrak{S}V\leq \xi F(t)\leq F(t)
\]
for any $t\geq 0$.

Then, appealing to the definition of $\xi(t)$ and performing
technical calculations, we have
\begin{align*}
\mathbf{D}_{t}^{\theta}\xi&=\frac{1}{\Gamma(1-\theta)}\int_{0}^{t}\frac{\xi'(\tau)}{(t-\tau)^{\theta}}d\tau=\frac{1}{\Gamma(1-\theta)}\begin{cases}
\int_{0}^{t}\frac{\xi'(\tau)}{(t-\tau)^{\theta}}d\tau,\qquad\quad\text{if}\qquad
t\in[0,2T^{*}],\\
\\
\int_{0}^{2T^{*}}\frac{\xi'(\tau)}{(t-\tau)^{\theta}}d\tau,\qquad\text{if}\qquad
t>2T^{*}
\end{cases}
\\
& \leq CT^{*1-\theta}\|\xi'\|_{\C(\overline{\R^{+}})}
\end{align*}
for any $\theta\in(0,1)$.

By means  of this inequality and h3, we deduce the bound
\[
|V_{1}|\leq
CV(0)\|\xi'\|_{\C(\overline{\R^{+}})}[T^{*1-\nu}+\sum_{i=1}^{M}T^{*1-\nu_{i}}][\|\rho\|_{L_{\infty}([0,+\infty))}
+\sum_{i=1}^{M}\|\rho_{i}\|_{L_{\infty}([0,+\infty))}],
\]
where the positive $C$ depends only on $\nu$ and $\nu_{i}$.

Coming to the term $V_{2}$, for any $\theta\in(0,1)$, we have the
easily verified relations:
\begin{align*}
\frac{\theta}{\Gamma(1-\theta)}\mathcal{J}_{\theta}(t;\xi,V)&=\frac{\theta}{\Gamma(1-\theta)}
\begin{cases}
\int_{0}^{t}\frac{[\xi(t)-\xi(s)][V(s)-V(0)]}{(t-s)^{1+\theta}}ds,\qquad\text{if}\quad
t\in[0,2T^{*}],\\
\\
\int_{0}^{2T^{*}}\frac{[\xi(t)-\xi(s)][V(s)-V(0)]}{(t-s)^{1+\theta}}ds,\qquad\text{if}\quad
t\geq 2T^{*},
\end{cases}
\\
\int_{0}^{t}\frac{[\xi(t)-\xi(s)][V(s)-V(0)]}{(t-s)^{1+\theta}}ds
&\leq
C\|\xi'\|_{\C(\overline{\R^{+}})}T^{*1-\theta}\big[\underset{t\in[0,2T^{*}]}{\sup}V(s)+V(0)\big],\quad\text{if}\quad
 t\in[0,2T^{*}],\\
\int_{0}^{2T^{*}}\frac{[\xi(t)-\xi(s)][V(s)-V(0)]}{(t-s)^{1+\theta}}ds
&\leq
C\|\xi'\|_{\C(\overline{\R^{+}})}\big[\underset{t\in[0,2T^{*}]}{\sup}V(s)+V(0)\big]\int_{0}^{2T^{*}}(2T^{*}-s)^{-\theta}ds\\
& \leq
CT^{*1-\theta}\|\xi'\|_{\C(\overline{\R^{+}})}[\underset{t\in[0,2T^{*}]}{\sup}V(s)+V(0)],
\quad\text{if}\quad
 t\geq 2T^{*}.
\end{align*}
Collecting these inequalities with  estimate \eqref{3.2} for
$t\in[0,2T^{*}]$, we conclude that  the bound holds
\[
|V_{2}|\leq
C[T^{*1-\nu}\|\rho\|_{L_{\infty}(\overline{\R^{+}})}+\sum_{i=1}^{M}T^{*1-\nu_{i}}\|\rho_{i}\|_{L_{\infty}(\overline{\R^{+}})}]\|\xi'\|_{\C(\overline{\R^{+}})}[1+V(0)]
\]
for any $t\geq 0$, where the constant $C$ is independent of $t$ and
$T^{*}$.

Concerning $V_{3}$, the properties of $\xi(t)$ provide the
representation
\[
V_{3}=
\begin{cases}
0,\qquad\qquad\qquad t\in[0,T^{*}],\\
\xi(t)(\mathcal{K}_{0}*V)(t)-(\mathcal{K}_{0}*\xi V)(t),\quad
t\in[T^{*},2T^{*}],\\
-\int_{0}^{2T^{*}}\mathcal{K}_{0}(t-\tau)\xi(\tau)V(\tau)d\tau,\quad
t>2T^{*}.
\end{cases}
\]
After that, appealing to the smoothness of $\mathcal{K}_{0}$ and
$\xi$ and using the estimate \eqref{3.2} for $t\in[0,2T^{*}]$, we
deduce
\[
|V_{3}|\leq
C\|\mathcal{K}_{0}\|_{L_{1}(0,+\infty)}[1+V(0)]\quad\text{for
any}\quad t\geq 0.
\]

In fine, collecting all estimates for $V_{i}$ and $\mathfrak{S}V$,
we end up with the desired bound, which completes the proof of
this claim.
\end{proof}
Now we introduce  new unknown function
\begin{equation}\label{5.12}
\mathcal{V}=V-\mathcal{U},
\end{equation}
which satisfies the relations
\[
\begin{cases}
\frac{\rho(t)}{2}\mathbf{D}_{t}^{\nu}\mathcal{V}+\sum_{i=1}^{M}\frac{\rho_{i}(t)}{2}\mathbf{D}_{t}^{\nu_{i}}\mathcal{V}-
\mathcal{K}_{0}*\mathcal{V}+C_{4}\mathcal{V}\leq
2F(t)+C[1+V(0)],\quad t\geq 0,\\
\mathcal{V}=0,\quad t\in[0,T^{*}].
\end{cases}
\]
Here we employed Proposition \ref{p5.1} and estimate \eqref{5.7}.

 Since the right-hand side in the inequality is strictly
positive continuous function for $t>0,$ we argue similar to
\cite[Theorem 1]{WX} and conclude that $\mathcal{V}$ solves the
Cauchy problem
\[
\begin{cases}
\frac{\rho(t)}{2}\mathbf{D}_{t}^{\nu}\mathcal{V}+\sum_{i=1}^{M}\frac{\rho_{i}(t)}{2}\mathbf{D}_{t}^{\nu_{i}}\mathcal{V}-\mathcal{K}_{0}*\mathcal{V}
+C_{4}\mathcal{V}=2F(t)+C[1+V(0)]+F_{1}(t),\quad t>0,\\
\mathcal{V}=0,\quad t\in[0,T^{*}],
\end{cases}
\]
where $F_{1}(t)$ is a nonnegative function.

 Obviously,
\[
F_1(t)=-2F(t)-C[1+V(0)]\quad\text{for}\quad t\in[0,T^{*}].
\]
    Finally, performing the change of the time-variable in the
    Cauchy
    problem
    \[
\sigma=t-T^{*},
    \]
and denoting
\[
\bar{\mathcal{V}}=\mathcal{V}(\sigma+T^{*}),
\]
we recast arguments of Section 6.5 in \cite{SV} and deduce
\begin{equation}\label{5.13}
\begin{cases}
\mathbf{D}_{\sigma}^{\nu}\bar{\mathcal{V}}+\sum_{i=1}^{M}d_{i}\mathbf{D}_{\sigma}^{\nu_{i}}\bar{\mathcal{V}}
-\frac{\mathcal{K}_{0}}{\rho(\sigma+T^{*})}*\bar{\mathcal{V}}+d_{0}\bar{\mathcal{V}}=F_{0}(\sigma)+\bar{F}_{1}(\sigma),\quad
\sigma>0,\\
\bar{\mathcal{V}}(0)=0,
\end{cases}
\end{equation}
where we set
\begin{align*}
d_{i}&=\frac{\rho_{i}(\sigma+T^{*})}{\rho(\sigma+T^{*})},\quad
i=1,2,..,M,\quad d_{0}=\frac{C_{4}}{\rho(\sigma+T^{*})},\\
F_{0}(\sigma)&=\frac{2F(\sigma+T^{*})}{\rho(\sigma+T^{*})}+\frac{C[1+V(0)]}{\rho(\sigma+T^{*})},\quad
\bar{F}_{1}(\sigma)=\frac{F_{1}(\sigma+T^{*})}{\rho(\sigma+T^{*})}.
\end{align*}
It is apparent (see h3 and the definition of $T^{*}$) that
$\rho(\sigma+T^{*})$ and $\rho_{i}(\sigma+T^{*}),$ $i=1,...,M,$
are positive constants for any $\sigma\geq 0$.

Thus, the initial-value problem \eqref{5.13} to \textbf{FODE} with
constant coefficients completes the \textit{Step 2}.

\noindent\textit{Step 3}. Here, we first aim to obtain  the
following bound of the solution to \eqref{5.13} for or all
$\sigma\geq 0$,
\begin{equation}\label{5.14}
\bar{\mathcal{V}}(\sigma)\leq
C[1+V(0)]+(G_{1}*\bar{\mathcal{V}})(\sigma)
\end{equation}
with the positive constant $C$ being independent of $\sigma$
 and the certain positive kernel  $G_{1}\in L_{1}(0,+\infty)$
 defined via $\mathcal{K}_{0}$.
 Then, keeping in mind the definition of the function
 $\bar{\mathcal{V}}$ and the estimate of the function $\mathcal{U}$, we will reduce
 \eqref{5.14} to the estimate
 \[
0\leq V(t)\leq C[1+V(0)]+G_{1}*|\mathcal{U}|+G_{1}*V.
 \]
In fine, exploiting the Gronwall-type inequality \cite[Theorem
15]{Dr} here, we reach
 the desired bound
 \[
0\leq V(t)\leq C[1+V(0)]\quad\text{for}\quad t\geq 0,
 \]
which completes the proof of Theorem  \ref{t3.1}.

Thus, we are left to verify \eqref{5.14}. To this end, applying
\cite[Theorem 4.1]{LG} to \eqref{5.13}, we have
\[
\bar{\mathcal{V}}=\int_{0}^{t}\mathcal{E}_{\bar{\beta},\beta_{0}}(t-\tau;\bar{d})F_{0}(\tau)d\tau
+\int_{0}^{t}\mathcal{E}_{\bar{\beta},\beta_{0}}(t-\tau;\bar{d})\bar{F}_{1}(\tau)d\tau
+
\int_{0}^{t}\mathcal{E}_{\bar{\beta},\beta_{0}}(t-\tau;\bar{d})\bigg(\frac{\mathcal{K}_{0}}{\rho^{*}}*\bar{\mathcal{V}}\bigg)(\tau)d\tau
\equiv \sum_{j=1}^{3}l_{j}
\]
where $\mathcal{E}_{\bar{\beta},\beta_{0}}(t;\bar{d})$ is defined
in \eqref{4.1} with
\[
\beta_{0}=\nu,\quad
\bar{\beta}=(\nu,\nu-\nu_1,...,\nu-\nu_{M}),\quad
\bar{d}=\bigg\{\frac{C_{4}}{\rho^{*}},\frac{\rho^{*}_{1}}{\rho^{*}},...,\frac{\rho^{*}_{M}}{\rho^{*}}\bigg\},\quad
\rho_{i}^{*}=\rho_{i}(\sigma+T^{*}),\quad
\rho^{*}=\rho(\sigma+T^{*}).
\]
It is worth noting that $\rho^{*}$ and $\rho_{i}^{*}$ are positive
constants for all $\sigma\geq 0$.

 After that, the desired estimate
follows immediately from this representation and Lemma \ref{lA.1}.
Indeed, statement \texttt{(i)} of Lemma \ref{lA.1} tells us that
$\mathcal{E}_{\bar{\beta},\beta_{0}}$ is a complete monotonic
function, which provides non-positivity of the second term $l_{2}$
(recalling that $\bar{F}_{1}\leq 0$).

 As for the term $l_1$, collecting the
representation of $F_{0}(t)$ (see \eqref{5.7*}) with statement (v)
of Lemma \ref{lA.1} arrives at the inequalities
\begin{align*}
|l_2|&\leq
C[1+V(0)][\mathcal{E}_{\bar{\beta},\nu}*1+\rho(0)\mathcal{E}_{\bar{\beta},\nu}*\omega_{1-\nu}+\sum_{i=1}^{M}\rho_{i}(0)\mathcal{E}_{\bar{\beta},\nu}*\omega_{1-\nu_{i}}]\\
& \leq
C[1+V(0)][1+\rho(0)(1+t^{\nu})^{-1}+\sum_{i=1}^{M}\rho_{i}(0)(1+t^{\nu_{i}})^{-1}]
\end{align*}
with the positive constant $C$ being independent of $t$ and
$T^{*}$.

Concerning the term $l_3$, the technical calculations and Lemma
\ref{lA.1} produce the equality
\[
l_3=\bar{\mathcal{V}}*G_{1},
\]
with
\[
G_{1}(t)=\int_{0}^{t}\frac{\mathcal{K}_{0}(t-z)}{\rho^{*}}\mathcal{E}_{\bar{\beta},\nu}(z;\bar{d})dz
\]
satisfying relations
\begin{align}\label{5.15}\notag
G_{1}&\in L_{1}(0,+\infty),\qquad G_{1}(t)\geq 0\quad\text{for all}\quad t\geq 0,\\
\|G_{1}\|_{L_{1}(0,+\infty)}&\leq
C\|\mathcal{K}_{0}\|_{L_{1}(0,+\infty)}\leq
C\|\mathcal{K}\|_{L_{1}(0,+\infty)}.
\end{align}
Here, we also used the definition of $\mathcal{K}_{0}$ (see
\eqref{5.7*}).

At last, gathering estimates of $l_{i}$, we end up with the bound
\[
\bar{\mathcal{V}}\leq
C[1+V(0)][1+\sum_{i=1}^{M}\max\{1,\rho_{i}(0)\}]+G_{1}*\bar{\mathcal{V}}.
\]
Substituting   representation \eqref{5.12} to  this estimate and
employing Proposition \ref{p5.1}, we have
\begin{align*}
V(t)&\leq
C(1+V(0))[1+\rho(0)+\sum_{i=1}^{M}\rho_{i}(0)]+G_{1}*V+G_{1}*\mathcal{U}\\
& \leq
C(1+V(0))[1+\|G_{1}\|_{L_1(0,+\infty)}+\rho(0)+\sum_{i=1}^{M}\rho_{i}(0)]+G_{1}*V
\end{align*}
for each $t\geq 0$.

Finally, taking into account nonnegativity of the function $V(t)$,
estimates \eqref{5.15} and exploiting the Gronwall-type inequality
\cite[Theorem 15]{Dr}, we obtain the bound
\[
V(t)\leq
C[1+V(0)][1+\|\mathcal{K}\|_{L_{1}(0,+\infty)}]\exp\{C\|\mathcal{K}\|_{L_{1}(0,+\infty)}\}
\]
for all $t\geq 0$. Here the positive value $C$ is independent of
$t$ and depends only on the given parameters in the model, and
corresponding norms of the coefficients.

Recalling the definition of the function $V$ in the estimate above,
we finishes the proof of \eqref{3.2} and accordingly  the proof of
Theorem \ref{t3.1}.\qed



\section{Proof of Theorem \ref{t3.2}}
\label{s6}

\noindent First of all we notice that the most arguments of Section
\ref{s5} (with minor modifications) are employed in the
multidimensional case. Thus, here we focus only on the main
differences between the multidimensional and the one-dimensional
cases. Obviously, these differences are related with study given in
the Step 1. Here we also focus on the \textbf{DBC} case
 and, following  \textit{Step 1} of the previous section, we
multiply the equation in \eqref{i.1} by
\[
v=u-\Delta u
\]
and integrate over $\Omega$. In this section (analogously to
Section \ref{s5}), we carry out only formal computations, as for
rigorous analysis, one should use \cite[Lemma 4.4]{KPSV5}.

Performing technical calculations, we have
\begin{equation}\label{6.1}
\sum_{i=1}^{4}\mathfrak{W}_{i}=0,
\end{equation}
where
\begin{align*}
\mathfrak{W}_{1}&=-\int_{\Omega}\bigg[\sum_{i=1}^{n}\frac{\partial^{2}
u}{\partial x_{i}^{2}}+a_{0}u\bigg]vdx,\qquad
\mathfrak{W}_{2}=-\int_{\Omega}\mathcal{K}*\big(\sum_{i=1}^{n}b_{i}\frac{\partial
u}{\partial x_{i}}+b_{0}u\big)vdx,\\
\mathfrak{W}_{3}&=\int_{\Omega}f(u)vdx,\qquad
\mathfrak{W}_{4}=\int_{\Omega}v\mathbf{D}_{t}u dx.
\end{align*}
Integrating by parts and taking into account regularity of the
coefficients (see h3 and h2) arrive at
\[
\mathfrak{W}_{1}\geq [1+\delta_0]\int_{\Omega}|\nabla
u|^{2}dx+\int_{\Omega}(\Delta
u)^{2}dx+\delta_{0}\int_{\Omega}u^{2}dx -
\int_{\Omega}\sum_{i=1}^{n}\frac{\partial a_{0}}{\partial
x_{i}}\frac{\partial u}{\partial x_{i}}udx.
\]
Then, exploiting the Cauchy inequality to handle the last term in
the right-hand side arrives at the bound
\[
\mathfrak{W}_{1}\geq \frac{1+\delta_{0}}{2}\int_{\Omega}|\nabla
u|^{2}dx+\int_{\Omega}(\Delta u)^{2}dx +
[\delta_{0}-\frac{2}{1+\delta_{0}}\|D_{x}
a_{0}\|^{2}_{L_{\infty}([0,+\infty),\C(\bar{\Omega}))}]\int_{\Omega}u^{2}dx.
\]
Coming to the evaluation of $\mathfrak{W}_{2}$ and
$\mathfrak{W}_{3}$, we repeat arguments employing to estimate the
terms $\mathcal{U}_{2}$ and $\mathcal{U}_{3}$ (see \textit{Step 1}
in the proof of Theorem \ref{t3.1}) and deduce the following bounds:
\begin{align*}
\mathfrak{W}_{2}&\geq
-8\sum_{i=0}^{n}\|b_{i}\|^{2}_{L_{\infty}([0,+\infty),\C(\bar{\Omega}))}\|\mathcal{K}\|_{L_{1}(0,+\infty)}
[\delta_{0}^{-1}+1]\int_{\Omega}|\mathcal{K}|*[u^{2}+|\nabla
u|^{2}]dx\\
&
-\frac{\delta_{0}}{8}\int_{\Omega}u^{2}dx-\frac{1}{8}\int_{\Omega}(\Delta
u)^{2}dx,\\
\mathfrak{W}_{3}&\geq
-C_{2}^{*}+\frac{2}{1+\delta_{0}}\|D_{x}a_{0}\|^{2}_{L_{\infty}([0,+\infty),\C(\bar{\Omega}))}\int_{\Omega}u^{2}dx
-\frac{1}{4}\int_{\Omega}(\Delta u)^{2}dx.
\end{align*}
Thus, collecting estimates of $\mathfrak{W}_{i},$ $i=1,2,3,$ ends up
with the inequality
\[
\sum_{i=1}^{3}\mathfrak{W}_{i}\geq
\frac{1+\delta_{0}}{2}\int_{\Omega}|\nabla u|^{2}dx +
\frac{5}{8}\int_{\Omega}(\Delta u)^{2}dx
+\frac{7\delta_{0}}{8}\int_{\Omega}u^{2}dx-C_{4}\int_{\Omega}|\mathcal{K}|*[u^{2}+|\nabla
u|^{2}]dx.
\]

As for the term $\mathfrak{W}_{4}$, we apply  arguments leading to
the estimate of $\mathcal{U}_{4}$ and obtain
\begin{align*}
\mathfrak{W}_{4}&\geq
\frac{1}{2}\bigg\{\rho(t)\mathbf{D}_{t}^{\nu}\int_{\Omega}(u^{2}+|\nabla
u|^{2})dx+\sum_{i=1}^{M}\rho_{i}(t)\mathbf{D}_{t}^{\nu_{i}}\int_{\Omega}(u^{2}+|\nabla
u|^{2})dx\bigg\}\\
&
-\|u_{0}\|^{2}_{W^{1,2}(\Omega)}\bigg\{\rho(0)\omega_{1-\nu}(t)+\sum_{i=1}^{M}\rho_{i}(0)\omega_{1-\nu_{i}}(t)+C_{3}\bigg\}.
\end{align*}
Thus, collecting all estimates of $\mathfrak{W}_{i}$ and using the
same notations to $C_{4}$ and $\mathcal{K}_{0}$, we get inequality
\eqref{5.7} to the function
\[
V(t)=\int_{\Omega}[u^{2}+|\nabla u|^{2}]dx.
\]
Then, to complete the proof of this claim, we just repeat
step-by-step all arguments of \textit{Steps 2-3} from  the proof of
Theorem \ref{t3.2}. \qed



\end{document}